\documentclass{elsarticle}

\usepackage{color}
\usepackage[TS1,T1]{fontenc}
\usepackage[latin1]{inputenc}
\usepackage{amsfonts,amssymb,amsmath,amsthm,enumerate,latexsym}
\usepackage{hyperref}
\usepackage{comment}
\usepackage{mathtools}
\mathtoolsset{showonlyrefs} 

\newtheorem{theorem}{Theorem}[section]
\newtheorem{lemma}{Lemma}[section]
\newtheorem{proposition}{Proposition}[section]

\theoremstyle{remark}
\newtheorem{remark}{Remark}

\theoremstyle{definition}
\newtheorem{definition}{Definition}[section]

\newtheorem{example}{Examples}

\numberwithin{equation}{section}

\begin{document}

\title{Weak and viscosity solutions for non-homogeneous fractional  equations in Orlicz spaces}
\author[ad1]{Mar\'ia L. de Borb\'on}
    \ead{laudebor@gmail.com}
\author[ad2]{Leandro M. Del Pezzo}
    \ead{ldpezzo@dm.uba.ar}
\author[ad3]{Pablo Ochoa}
    \ead{pablo.ochoa@ingenieria.uncuyo.edu.ar}
    
    \address[ad1]{Facultad de Ciencias Econ\'omicas, FCE, Universidad Nacional de Cuyo-CONICET, Parque Gra. San Mart\'in SN (5500), Mendoza, Argentina.}

    \address[ad2]{Departamento  de Matem{\'a}tica, FCEyN, Universidad de Buenos Aires, Pabellon I, Ciudad Universitaria (C1428BCW), Buenos Aires, Argentina.}

    \address[ad3]{Facultad de Ingenier\'ia, Universidad Nacional de Cuyo-CONICET, Parque Gra. San Mart\'in SN (5500), Mendoza, Argentina.}

\begin{abstract}
    In this paper, we consider non-homogeneous fractional equations in Orlicz spaces, with a source depending on the spatial variable, the unknown function and its fractional gradient. The latter is adapted to the Orlicz framework. The main contribution of the article is to establish the equivalence between weak and viscosity solutions  for such equations.
\end{abstract}

\begin{keyword}
    \MSC 35D40,\sep 35D30,\sep 35R11,\sep 46E30
    
    viscosity solutions, weak solutions, non-homogeneous problems, g-Laplace 
    operator, Orlicz spaces.
\end{keyword}

\maketitle

\section{Introduction}

    In this paper, 
    we prove the equivalence between weak and viscosity solutions 
    for the non-homogeneous fractional $g-$Laplace equation
    \begin{equation}\label{problem}
	    (-\Delta_g)^s u = f(x,u,D_g^su) \text{ in } \Omega,
    \end{equation}
    where $\Omega\subset \mathbb{R}^n$ is a bounded open domain. Given
    $s\in(0,1)$ and a Young function $G$ such that $g\coloneqq G^\prime$, the fractional $g-$Laplace operator
    is defined by
    
    \begin{equation}\label{defglaplacian}
            (-\Delta_g)^s u (x) = \text{P.V.} \int_{\mathbb{R}^{n}}g\left(\frac{u(x)-u(y)}{|x-y|^s}\right)\frac{dy}
            {|x-y|^{n+s}}
    \end{equation}
    for any smooth function $u.$ 
    Here P.V. is an abbreviation for ``in the principal value sense''. In 
    \eqref{problem}, $D_{g}^{s}$ denotes the $g-$fractional gradient of $u\in W^{s, G}(\Omega)$, that is defined as follows
    \[
        D_{g}^{s}u(x)\coloneqq 
        \int_{\mathbb{R}^{n}}G\left(\dfrac{|u(x)-u(y)|}{|x-y|^{s}}\right)\dfrac{dy}{|x-y|^{n}}, \quad x \in \Omega.
    \]
    Observe that $D_{g}^{s}u(x)$ is finite a.e. 
    since $u\in W^{s, G}(\Omega).$  
    We refer the reader to Section \ref{Preli} for 
    for definitions and properties of the spaces that we use.
    
    To state the equivalence of solutions, we will assume that $f\colon\Omega\times\mathbb{R}\times\mathbb{R}\to \mathbb{R}$ satisfies the following growth conditions   
    \begin{equation}\label{growth f}
        |f(x, r, \eta)| \leq \gamma(|r|)\tilde{G}^{-1}(|\eta|) + \phi(x),
    \end{equation}where $\tilde{G}$ is the complementary function of $G$,   
    $\gamma \geq0$ is continuous, and $\phi \in L^{\infty}(\Omega)$.

    \medskip
    
    If we take $g(t)=t$, then the fractional $g-$Laplace operator is the fractional Laplace operator, that is
    \[
        (-\Delta)^s u (x) = \text{P.V.} \int_{\mathbb{R}^{n}}\frac{u(x)-u(y)}{|x-y|^{n+2s}}dy
    \]
    This operator is the most basic elliptic linear integro-differential operator.
    Problems with non local diffusion that involve integro-differential operators 
	have been intensively studied in the last years. These nonlocal operators 
	appear when we model different physical situations as anomalous diffusion and 
	quasi-geostrophic flows, turbulence and water waves, molecular dynamics and 
	relativistic quantum mechanics of stars (see \cite{BG,C} and 
	{the references therein}). They also appear  in mathematical finance 
	\cite{A,CT}, elasticity  problems \cite{Si}, 
	phase transition problems  \cite{AB} and crystal dislocation structures
	\cite{T}, among others.
	
	On the other hand, the $g-$fractional gradient is the natural extension of 
	the fractional gradient
    \[
        D^{s}u(x)\coloneqq
        \int_{\mathbb{R}^{n}}\dfrac{|u(x)-u(y)|^2}{|x-y|^{n+2s}}dy, \quad x \in 
        \Omega.
    \]
    This $s-$gradient appears naturally when studying fractional harmonic maps to 
    the sphere. See \cite{CD} and the references therein.
    
    \medskip
    
    Our first main result is that viscosity solutions of \eqref{problem}
    are also weak solutions. We state the result for supersolutions. See Section \ref{Preli} for details in the assumptions below.

    \begin{theorem}\label{visc wk} 
        Assume that $G$ is a Young function satisfying \eqref{H1} and so that its 
        complementary $\tilde{G}$ verifies the $\Delta'$-condition. 
        Also, suppose that $f=f(x, r, \eta)$ is non-increasing in $r$,  
        uniformly continuous in $\Omega \times \mathbb{R}\times 
        \mathbb{R}$, Lipschitz continuous in $\eta$, and it  satisfies the growth \eqref{growth f}. If $u \in 
        L^\infty(\mathbb{R}^n)$ is lower semicontinuous in $\mathbb{R}^n$ and is  a viscosity supersolution 
        of \eqref{problem}, then $u$ is a weak supersolution of \eqref{problem}.  
    \end{theorem}

    \medskip
    
    {For the converse result we need to assume that a comparison principle holds 
    for weak solutions of \eqref{problem}. Following \cite{BM}, we define below the
    class of functions that satisfy the comparison principle property.} 

    \begin{definition}
        {Let $u$ be a weak supersolution of \eqref{problem} in $D\subset \Omega$. 
        We say that the comparison principle property (CPP) holds in $D$ if for 
        every weak subsolution $v$ of \eqref{problem} in $D$ such that $u\geq v$ 
        a.e. in $\mathbb{R}^n\setminus D$, we have $u\geq v$ a.e. in $D$.}
    \end{definition}

    {Next we state the reverse result.}

    \begin{theorem}\label{wk visc}
        {Assume that $G$ is a Young function satisfying \eqref{H1} and so that its 
        complementary $\tilde{G}$ verifies the $\Delta'$-condition.} 
        Also, suppose that $f=f(x, r, \eta)$ is  continuous in $\Omega \times 
        \mathbb{R}\times \mathbb{R}$ and Lipschitz continuous in $\eta$.
        If $u \in W^{s, G}(\mathbb{R}^n)\cap L_g^s(\mathbb{R}^n)$ is 
        a bounded weak supersolution of \eqref{problem} and (CPP) holds, 
        then $u$ is also a viscosity supersolution of \eqref{problem}.  
    \end{theorem}
    
    \medskip

    The relation between different notions of solutions has been studied by several
    authors and for different operators in the last decades. For linear problems,
    the equivalence between distributional and viscosity solutions is given in
    \cite{Is}. Later, for weak and viscosity solutions for the homogeneous 
    $p-$Laplace equations, the equivalence was provided in \cite{JLM1} 
    (with a different proof in \cite{JJ}). Recently, for a source depending on all 
    the lower-order terms, the relation between weak and viscosity solutions for 
    the $p-$Laplace equation was given in \cite{MO}, following some ideas from 
    \cite{JJ}. Similar studies have been made for operator with non-standard growth
    \cite{JLM, Sil}, and  recently for non-local operators \cite{KKL} and 
    \cite{BM}. In this work, we propose to generalize the later results to the 
    non-homogeneous fractional $g-$laplacian operator.

    Applications of the equivalence between viscosity and weak solutions can be 
    found in  \cite{Sil} and \cite{JL} to removability of sets and Rado type 
    theorems. Also, the equivalence has been used in  free-boundary problems 
    \cite{FL}, \cite{BLO}.
    
    \medskip

    {\bf The paper is organized as follows}. In Section \ref{Preli}, we give the  
    definition of the Orlicz spaces, some technical results and the notions of 
    solutions. In Section  \ref{previous results}, we give previous results for 
    viscosity and weak solutions that we shall use to state the equivalence of 
    solutions. {There we also provide some continuity properties necessary for the proof of 
    Theorem \ref{wk visc}}. Afterwards, in Sections \ref{visc wk t} and \ref{wk visc t} we 
    prove the main results of the paper. We end the paper with a brief appendix regarding some
    inequalities for Young functions.

\section{Preliminaries}\label{Preli}
    In this section, we gather some preliminary properties which will be useful in the forthcoming sections.

\subsection{Notations}  Throughout the paper, we  use the notation
    \[
        u_+(x)=\max\{u(x),0\}, \quad u_-(x)=\max\{-u(x),0\}, 
    \]
    \[
        D_s u \coloneqq \dfrac{u(x)-u(y)}{|x-y|^{s}}
        \quad \text{ and } \quad 
        d\mu\coloneqq\dfrac{dxdy}{|x-y|^{n}}.
    \]
    
    For all $a\in \mathbb{R}$ and $q>0,$ we set
    \[
        a^q=|a|^{q-1}a.
    \]
\subsection{Young functions}
    
    A  function $G:[0, \infty) \to [0, \infty)$ is called 
    {\it a Young function} if it has the following
     integral representation
    \[ 
        G(t)=\int_0^{t}g(r)\,dr
    \]
    where the right-continuous function $g: [0, \infty) \to [0, \infty)$ satisfies
    \begin{enumerate}   
        \item $g(0)=0$, $g(t)>0$ for $t>0$;\vspace{.25cm}
        \item $g$ is non-decreasing in $[0, \infty)$;\vspace{.25cm}
        \item $\lim\limits_{t \to \infty}g(t)=\infty$.
    \end{enumerate}
    {Observe that we use the term \textit{Young function} to denote N-functions as defined in 
    \cite{KR}}.
    
    From the above properties, it is easy to see that a Young function $G$ is 
    continuous, nonnegative, strictly increasing, and convex in $[0, \infty)$. 
    Also, without loss of generality we can assume that $G(1) = 1$.
    
    \medskip

    In this article, we consider Young functions $G$ assuming that  
    $g=G'$  is an absolutely continuous function such that 
    \begin{equation}\label{H1}
        p^--1 \leq \dfrac{tg'(t)}{g(t)} \leq p^+-1, \quad\forall t>0
    \end{equation}
    with $1< p^-<p^+<\infty.$ 
    
    \begin{example}
       The following functions are Young functions that satisfy \eqref{H1}:
       \begin{enumerate}[(a)]
           \item $G(t)= t^p$ with $p>1.$
           \item $G(t)=t^p(|\log(t)|+1)$ with $p>\tfrac{3+\sqrt{5}}{2}.$
       \end{enumerate} 
    \end{example}
    
    Observe that, by integration by parts, \eqref{H1} gives
    \begin{equation}\label{H11}
        1< p^{-} \leq \dfrac{tg(t)}{G(t)} \leq p^{+} <\infty\quad\forall t>0.
    \end{equation}
    Therefore, see \cite[Theorem 4.1]{KR}, $G$ satisfies the $\Delta_2-$condition,
    that is, there exists a positive constant $C$ such that
    \[
        G(2t)\le CG(t)
    \]
    for any $t\ge0.$ Moreover, from \eqref{H1} and  \eqref{H11} it follows that
   \begin{equation}\label{gg product}
         {\min\left\lbrace a^{p^{-}-1}, a^{p^{+}-1} \right\rbrace g(b) \leq g(ab) \leq  \max\left\lbrace a^{p^{-}-1}, a^{p^{+}-1} \right\rbrace g(b),}
    \end{equation}and
    \begin{equation}\label{G product}
        \min\left\lbrace a^{p^{-}}, a^{p^{+}} \right\rbrace G(b) \leq G(ab) \leq  \max\left\lbrace a^{p^{-}}, a^{p^{+}} \right\rbrace G(b),
    \end{equation}for all $a, b \geq 0$.

    \medskip
    
    On the other hand, we may extend $g$ to the whole $\mathbb{R}$ as following: 
    \begin{equation*}
            g(t) = -g(-t)\quad \text{if }\; t<0.
    \end{equation*}
    Then, \eqref{H1} also holds for negative $t$, and it   implies that there is a positive constant $C$ such that
    \begin{equation}\label{H111}
        C\min\{t^{p^{-}-1},t^{p^+-1}\}\leq g(t)\leq
        C\max\{t^{p^{-}-1},t^{p^+-1}\},\quad t\in\mathbb{R}.
    \end{equation}
    
    \medskip 
    
    {By \cite[Lemma 1.3]{KR}, $G$ is an absolutely continuous function. Then, applying  \cite[Lemma 2.2]{MSV} to $G$, and using \eqref{H11} and Lemma \ref{tineqg} there is 
    a positive constant $C$ such that }
    \begin{equation}\label{inq G and g}
        |G(b)-G(a)| \leq C|b-a|(g(|a|)+ g(|b|)), \quad {\forall a,b\in\mathbb{R}}.
    \end{equation}
    
    \medskip
    
    We say that a Young function $G$ satisfies the $\Delta'$-condition if there exists a
    positive constant $C\geq 1$ such that
    \begin{equation}\label{delta prime}
         G(ab) \leq CG(a)G(b)
    \end{equation}
    for all $a, b \geq 0$. 
    Observe that if $G$ satisfies the $\Delta'$-condition, then it also satisfies the $\Delta_2$-condition.

    Examples of functions satisfying the $\Delta'$-condition are:
    \begin{itemize}
        \item $G(t)= t^{p}$, $t \geq 0$, $p > 1$;
        \item $G(t)= t^{p}(|\log(t)|+1)$, $t \geq 0$, $p > \tfrac{3+\sqrt{5}}{2}$;
        \item $G(t)=t^{p}\chi_{(0, 1]}(t) + t^{q}\chi_{(1, \infty)}(t)$, $t \geq 0$, $p, q > 1$.
    \end{itemize}

    Necessary and sufficient conditions for the $\Delta'$-condition are given in \cite[Chapter I, Sec. 5]{KR}.

    \medskip
    
    The {\it complementary function} of a Young function $G$ is defined 
    on $[0,\infty)$ by
    \[
        \tilde{G}(a)\coloneqq \sup\left\lbrace  at-G(t)\colon t> 0\right\rbrace.
    \]
    
    The function $\tilde{G}$ plays the role that  the conjugate function exponent has in the standard theory of Lebesgue and Sobolev spaces. {Also, by \cite[Chapter I, Theorem 4.3]{KR}, the  inequality \eqref{H11} implies that $\tilde{G}$ satisfies the $\Delta_2$-condition. Moreover, by \cite[Chapter I, Theorem 5.3]{KR}, $\tilde{G}$ satisfies the $\Delta'$-condition if the function $h(t):=tg'(t)/g(t)$ does not decrease.}
    
    As a consequence of the inequality \cite[Eq. (2.5)]{BS}
    \[
        at \leq G(t)+\tilde{G}(a), \quad \text{for any } t, a\geq 0,
    \]
    and the inequalities (from \eqref{G product}),
    \begin{align}
        \label{ineqGab} G(a t) \leq tG(a), & \text{ for any } 
            {a\geq 0, \,0 \leq t\leq 1}, \\[.25cm]
        \label{ineqGwithp} G(at) \leq t^{p^+}G(a), 
        & \text{ for any } \, {a\geq 0},  t\geq 1,
    \end{align}
    we get the  following Young's inequality   for $0 < \delta<  1$
 
    \begin{equation}\label{Young}
        at = (a\delta)\left(\frac{1}{\delta}t\right) \leq \tilde{G}(\delta a) + G\left(\frac{1}{\delta}t\right) \leq \delta \tilde{G}(a) + 
        \left(\frac{1}{\delta}\right)^{p^+}G(t)
        \quad \forall a,t>0.
    \end{equation}

    Finally, we quote the following useful lemma.
    \begin{lemma}\cite[Lemma 2.9]{BS}\label{G g}
        Let $G$ be an Young function. If $G$ satisfies \eqref{H1} then 
        \[
            \tilde{G}(g(t)) \leq (p^+-1)G(t),
        \]
        where $g=G'$ and $\tilde{G}$ is the complementary function of $G.$
\end{lemma}

We will provide some more useful inequalities for Young functions in the Appendix.

\subsection{Orlicz-fractional Sobolev spaces}
    Given a Young function $G$ with $g=G'$, $s \in (0, 1)$, 
    and an open set $\Omega \subseteq \mathbb{R}^{n}$, 
    we consider the spaces:
    \[
        \begin{split}
            & L^{G}(\Omega) \coloneqq 
            \left\lbrace u\colon \Omega \to \mathbb{R} \colon 
            \Phi_{G, \Omega}(u) < \infty \right\rbrace,\\[5pt]
            &W^{s, G}(\Omega)\coloneqq 
            \left\lbrace u \in L^{G}(\Omega)\colon 
            \Phi_{s, G, \Omega}(u) < \infty \right\rbrace, \text{ and }\\[5pt]
            &L_g(\mathbb{R}^n)\coloneqq\left\lbrace u \in 
            L^{1}_{loc}(\mathbb{R}^{n})\colon 
            \int_{\mathbb{R}^{n}}g\left(\dfrac{|u(x)|}{1+|x|^{s}}\right)\dfrac{dx}
            {1+|x|^{n+s}} < \infty\right\rbrace.
        \end{split}
    \]
    Here, the modulars $\Phi_{G, \Omega}$ and $\Phi_{s, G, \Omega}$ are defined as
    \[    
        \Phi_{G, \Omega}(u) \coloneqq \int_{\Omega}G(|u(x)|)\,dx, 
        \quad\text{ and }\quad
        \Phi_{s, G, \Omega}(u) \coloneqq 
        \int_{\Omega}\int_{\Omega}
        G\left(|D_su|\right)
            d\mu.
    \]
    
    The spaces $L^G(\Omega)$ and $W^{s,G}(\Omega)$ are endowed, 
    respectively, with the following norms
    \[
        \|u\|_{L^G(\Omega)} \coloneqq 
        \inf\left\{\lambda >0 \colon 
        \Phi_{G,\Omega}\left(\frac{u}{\lambda}\right)\leq 1\right\},
    \]
    and 
    \[
         \|u\|_{W^{s,G}(\Omega)} \coloneqq 
         ||u||_{L^G(\Omega)} + [u]_{W^{s,G}(\Omega)},
    \]
    where 
    \[
        [u]_{W^{s,G}(\Omega)} \coloneqq \inf\left\{\lambda>0 \colon
        \Phi_{s,G,\Omega}\left(\frac{u}{\lambda}\right)\leq 1 \right\}.
    \]
    
    {The space $W^{s,G}_0(\Omega)$ will be the closure of $C^\infty_c(\Omega)$ with respect to
    the norm $\|\cdot\|_{W^{s,G}(\Omega)}$.}

    \medskip
    
    The following lemma relates modulars and norms in Orlicz spaces.  

    \begin{lemma}\label{comp norm modular}
        Let $G$ be a Young function satisfying \eqref{H1}, and let 
        $\xi^\pm\colon[0,\infty)\to\mathbb{R}$ be defined as
        \[
            \xi^{-}(t)\coloneqq 
            \min\left\lbrace t^{p^{-}}, t^{p^{+}}\right\rbrace,
            \quad \text{ and }  \quad
            \xi^{+}(t)\coloneqq\max\left\lbrace t^{p^{-}}, t^{p^{+}}\right\rbrace. 
         \] 
         Then
        \begin{enumerate}
            \item $\xi^{-}(\|u\|_{L^{G}(\Omega)}) \leq \Phi_G(u) 
                    \leq  \xi^{+}(\|u\|_{L^{G}(\Omega)})$;
            \item $\xi^{-}([u]_{W^{s, G}(\Omega)}) 
                \leq \Phi_{s, G}(u) \leq  
                \xi^{+}([u]_{W^{s, G}(\Omega)}).$
        \end{enumerate}
    \end{lemma}

\subsection{Notions of Solutions}
    Borrowing ideas from \cite{KKL}, 
    we first introduce the definition of  viscosity solution. For a given domain $\Omega \subset \mathbb{R}^n $ and $\beta >2$, we let
    
%
    \begin{equation*}
        C^2_\beta(\Omega)\coloneqq \left\{u\in C^2(\Omega)\colon \sup_{x\in \Omega}\left(\frac{\min\{d_u(x),1\}^{\beta-1}}{|\nabla u(x)|}+\frac{|D^2u(x)|}{d_u(x)^{\beta-2}}\right) < \infty\right\}
    \end{equation*}where $d_u(x) \coloneqq \text{dist}(x,N_u)$ and $N_u: = \left\{x\in \Omega \,:\, \nabla u(x) = 0\right\}$ represents the set of critical points of the function $u$. 
    \begin{definition}\label{dv}
	    We say that a function $u:\mathbb{R}^n\to[-\infty,+\infty]$ is a viscosity supersolution 
        (subsolution) of \eqref{problem} in $\Omega$ if: 
	    \begin{enumerate}
		    \item [(i)] $u<+\infty$ ($u>-\infty$) a.e. in $\mathbb{R}^n$, $u>-\infty$ ($u<+\infty$) 
		    a.e. in $\Omega$,
		    \item [(ii)] $u$ is lower (upper) semicontinuous in $\Omega$,
		    \item [(iii)] if $\psi\in C^2(B_r(x_0))\cap L_g(\mathbb{R}^n)$ for some 
		    $B_r(x_0)\subset \Omega$ such that $\psi(x_0)= u(x_0)$ and $\psi\leq u$ ($u\leq \psi$) in 
		    $\mathbb{R}^n$ and one of the following holds:
		    \begin{enumerate}[(a)]
		        \item $p^{-}>\frac{2}{2-s}$ or $\nabla \psi(x_0)\neq 0$,
		        \item $1<p^{-}\leq \frac{2}{2-s}$, $\nabla \psi(x_0)=0$ such that $x_0$ is an isolated 
                critial point of $\psi$ in $B_r(x_0)$, and $\psi\in C^2_\beta(B_r(x_0))$ for some 
                $\beta>\frac{sp^{-}}{p^{-}-1},$
		    \end{enumerate}
            then
		    \begin{equation*}
		        (-\Delta_g)^s\psi(x_0)\geq f(x_0,\psi(x_0),D_g^s\psi(x_0)),
		    \end{equation*}
		        \item [(iv)] $u_{-}$ ($u^+$) belongs to 
		        $L_g(\mathbb{R}^n)$.
	        \end{enumerate}
	        A viscosity solution of \eqref{problem} is a function $u$ which is a viscosity sub- and 
            supersolution of \eqref{problem}.
    \end{definition}
    \begin{remark}\label{u bdd}
        Observe that when $u \in L^\infty(\mathbb{R}^n)$, in Definition \ref{dv} (iii), we may define $\psi$ as $u$ outside the ball $B_r(x_0)$.
    \end{remark}
    \begin{remark}
        In Section \ref{previous results}, we shall show that $(-\Delta_g)^s\psi$ is well-defined for the class of test functions considered above.
    \end{remark}

 We now give the definition of weak solutions. 
\begin{definition}\label{dw}
    A function $u \in W^{s, G}(\Omega) \cap L_g(\mathbb{R}^{n})$ is 
    {\it a weak supersolution (subsolution) of \eqref{problem} }if
    \[
        \int_{\mathbb{R}^{n}}\int_{\mathbb{R}^{n}}
            g\left(D_su \right)
        D_s\psi
        \,d\mu \geq\,(\leq)\, \int_{\Omega}f(x, u, D_g^{s}u)\psi\,dx,
    \]
    for any non-negative $\psi \in C^{\infty}_0(\Omega)$. We say that $u$ is a weak solution if 
    it is a weak sub- and supersolution. 
 \end{definition}  
 
 Observe that by density, we may extend the above definition to test functions in $W_0^{s, G}(\Omega)$.
 
\subsection{Infimal convolutions}
    A standard smoothing operator in the theory of viscosity solutions is the infimal convolution. 
    
\begin{definition}\label{definfconv}
	Given $\varepsilon>0$, we define the infimal convolution of a function 
	$u\colon\mathbb{R}^n\to \mathbb{R}$ as 
	
	$$u_\varepsilon(x): = 
	\inf_{y\in\mathbb{R}^n}\left(u(y)+\frac{|x-y|^q}{q\varepsilon^{q-1}}\right)$$
	where $q=2$ if $p^{-}>\frac{2}{2-s}$ and $q>\frac{sp^{-}}{p^{-}-1}\geq 2$
	if $1<p^{-}\leq \frac{2}{2-s}$.
\end{definition}
    The infimal convolution is one of the main tools to prove that any
    viscosity solution is a weak solution. See, for instance \cite{BM,JJ}
    and the references therein.

\begin{lemma}[See \cite{JJ}]\label{propinfconv}
	Let $u$ be a bounded and lower semicontinuous function in $\mathbb{R}^n$.
	Then: 
	\begin{enumerate}[(i)]
		\item There exists $r(\varepsilon)>0$ such that 
		\[
		    u_\varepsilon(x) = \inf_{y\in B_{r(\varepsilon)}(x)}\left(u(y)+\frac{|x-y|^q}{q\varepsilon^{q-1}}\right)
		\]
		where $r(\varepsilon)\to 0$ as $\varepsilon\to 0$.
		
		\item The sequence $\{u_\varepsilon\}_{\varepsilon>0}$ is increasing as $\varepsilon\to 0$ and $u_\varepsilon\to u$ pointwise in $\mathbb{R}^n$.
		\item $u_\varepsilon$ is locally Lipschitz and twice differentiable a.e. Actually, for almost every $x,y\in \mathbb{R}^n$
		\[
		    u_\varepsilon(y) = u_\varepsilon(x)+\nabla u_\varepsilon(x)\cdot(x-y)+\frac{1}{2}D^2u_\varepsilon(x)(x-y)^2
		    +o(|x-y|^2) .
		\]

		\item $u_\varepsilon$ is semiconcave, that is, there exists a constant $C=C(q,\varepsilon,\text{osc}(u))>0$ such that the function $x\mapsto u_\varepsilon(x)-C|x|^2$ is concave. In particular 
		\[
	    	D^2u_\varepsilon(x)\leq 2CI,\quad \text{a.e. }\, x\in\mathbb{R}^n.
	    \]
		
		\item The set $Y_\varepsilon(x)\coloneqq\left\{y\in     B_{r(\varepsilon)}(x)\colon u_\varepsilon(x)=u(y)+\frac{|x-y|^q}{q\varepsilon^{q-1}} \right\}$ is non empty and closed for every $x\in\mathbb{R}^n$.
		
		\item If $\nabla u_\varepsilon(x) = 0$, then $u_\varepsilon(x) = u(x)$.
	\end{enumerate}
\end{lemma}
 
 \section{Previous results}\label{previous results}
     In this section we shall provide preliminary results 
     to state the equivalence between weak and viscosity solutions of the equation
     \eqref{problem}.  {We divide the section into three parts: results related to viscosity 
    solutions, results regarding the weak formulation of 
    solutions and the last part is dedicated to certain continuity properties of $D_g^s$ and $(-\Delta_g)^s$.} {To prove the results, we borrow some calculations from \cite{KKL} and  \cite{BM}. However, due to some technical differences, we prove the results in detail.}
 
 \subsection{Results on viscosity solutions} 
    In this section, we prove that $(-\Delta_g)^s \psi$ is well-defined for the 
    test functions introduced in Definition \ref{dv}. Moreover, we will state the equations satisfied by the inf-convolution of a viscosity solution.
    
    We start with some 
    preliminary lemmas. 
    
	\begin{lemma}\label{lemma:PV1}
	    Let $\rho >0$ be such that $B_\rho (x)\subset D\subset\subset \Omega$ 
	    (with $D$ open),
	    and $u\in C^2(D).$ If $p^->\tfrac{2}{2-s}$ or 
	    $D\subset\subset\{d_u>0\}$
	    then there is  a positive constant $C_\rho$ independent of $x$ such that
	    \[
	        \left|\textnormal{P.V.} 
	        \int_{B_\rho(x)}g\left(D_su\right)
	        \frac{dy}{|x-y|^{n+s}}\right|\le C_\rho.
	    \]
	    Moreover $C_\rho\to0$ as $\rho\to 0^+.$
	\end{lemma}
	\begin{proof}
	    {\it Case 1:} $\nabla u(x)=0$ and $p^{-}>\tfrac{2}{2-s}.$
	    
	    Since $u\in C^2(D),$ and $\nabla u(x)=0,$ for a given $\rho_0 >0$ so that $B_{\rho_0}(x) \subset D$,  there is a positive constant $C$
	    independent of $x$ and $y$ such that
	    \[
	        |u(x)-u(y)|\le C|x-y|^2
	    \]
	    for all  $y \in B_{\rho_0}(x)$. Then, for all $\rho \leq \rho_0$ so that $C\rho^{2-s} \leq 1$, we have
	    $$|D_s u| \leq C|x-y|^{2-s}\leq C\rho^{2-s} < 1.$$Hence, by \eqref{H111} that
	    \[
	        \left|g\left(D_s u\right)\right|\le 
	        g\left(D_s u \right)\le
	        \left(D_s u\right)^{p^--1}
	        \le C|x-y|^{(2-s)(p^--1)}.
	    \]
	    Therefore, we have
	    \[
	        \left|\text{P.V.} \int_{B_\rho(x)}g\left(D_s u\right)
	        \frac{dy}{|x-y|^{n+s}} \right|\le C \int_{B_\rho(x)}|x-y|^{(2-s)p^{-}-n-2} dy=C\rho^{(2-s)p^{-}-2}.
	    \]
	    Using that $p^->\tfrac{2}{2-s},$ it follows that
	    \[
	        C_\rho\coloneqq C\rho^{(2-s)p^{-}-2} \to 0,
	    \]
	    as $\rho \to 0$.

	    \medskip
	    
	    \noindent {\it Case 2:} $\nabla u(x)\neq 0.$
	    
	    Let $L(y)\coloneqq u(x)+\nabla u(x)\cdot (y-x).$ Observe that, since $g$ is odd, we
	    get
	    \[
	        \text{P.V.} \int_{B_\rho(x)\setminus B_\varepsilon(x)}g
	        \left(D_s L\right)
	        \frac{dy}{|x-y|^{n+s}}=0
	    \]
	    for any $\varepsilon<\rho.$ Then, for any $0<\varepsilon<\rho,$ taking
	    $B_{\varepsilon,\rho}\coloneqq B_\varepsilon(x)\setminus B_\rho(x)$ we have
	    \[
	         \left|\int_{B_{\varepsilon,\rho}}
	         g\left(D_s u\right)
	        \frac{dy}{|x-y|^{n+s}}\right|\leq 
	        \int_{B_{\varepsilon,\rho}}
	        \left|g\left(D_s u\right)-
	            g\left(D_s L\right)\right| \frac{dy}{|x-y|^{n+s}}.
	    \]
	    Thus, by Lemma \ref{lema:aux1}, there is a positive constant $C$ independent of
	    $x,\varepsilon,$ and $\rho$ such that
	    \begin{align*}
	        &\left|\int_{B_{\varepsilon,\rho}}
	        g\left(D_s u\right)
	        \frac{dy}{|x-y|^{n+s}}\right|\le\\
	        &\le C\int_{_{B_{\varepsilon,\rho}}}
	        \max\left\{ \dfrac{H(u,L,x,y)^{p^--2}}{|x-y|^{s(p^--2)}},  \dfrac{H(u,L,x,y)^{p^+-2}}{|x-y|^{s(p^+-2)}}\right\}
	        \dfrac{|u(y)-L(y)|}{|x-y|^{n+2s}}dy.
	    \end{align*}
	    where $H(u,L,x,y)\coloneqq |L(x)-L(y)|+|u(y)-L(y)|.$
	    Then 
	    \begin{align*}
	        &\left|\int_{B_{\varepsilon,\rho}}
	        g\left(D_s u\right)
	        \frac{dy}{|x-y|^{n+s}}\right|\le\\
	        &\le C\left\{\int_{_{B_{\varepsilon,\rho}}}
	        \dfrac{H(u,L,x,y)^{p^--2}}{|x-y|^{N+sp^-}}|u(y)-L(y)| dy
	        +\int_{_{B_{\varepsilon,\rho}}}
	        \dfrac{H(u,L,x,y)^{p^+-2}}{|x-y|^{N+sp^+}}|u(y)-L(y)| dy
	        \right\}.
	    \end{align*}
	    The rest of the proof follows the same lines as that of \cite[Lemma 3.6]{KKL}.
	\end{proof}

    \begin{lemma}\label{lemma:PV2}
        Let $1<p^-\le\tfrac{2}{2-s},$ $D\subset\Omega$ be an open set, and 
        $u\in C^2_{\beta}(D)$ with $\beta>\tfrac{sp^-}{p^--1}.$Then, for any 
        $\rho \in(0,1)$ such $B_\rho(x)\subset D$ and $x$ is such that $d_u(x)<\rho.$
        Then there is a positive constant $C_\rho$ independent of $x$ such that
        \[
	        \left|\text{P.V.} \int_{B_\rho(x)}g\left(D_s u\right)\frac{dy}{|x-y|^{n+s}}\right|\le C_\rho.
	    \]
	    Moreover $C_\rho\to0$ as $\rho\to 0^+.$
    \end{lemma}
    \begin{proof}
        {\it Case 1:} $\nabla u(x)=0.$ 
	    
	    Since $u\in C^2_\beta (D),$ and $\nabla u(x)=0,$ there is a positive constant $C$
	    independent of $x$ such that
	    \[
	        |u(x)-u(y)|\le C|x-y|^\beta
	    \]
	    in $B_\rho(x),$ for all $\rho$ small enough.  Then, as in the proof of Lemma \ref{lemma:PV1}, using \eqref{H111}, we have that
	    \[
	        \left|g\left(D_s u\right)\right|\le 
	        g\left(D_s u\right)\le
	        \left(D_s u\right)^{p^--1}
	        \le C|x-y|^{(\beta-s)(p^--1)}.
	    \]
	    Therefore, using that  $\beta>\tfrac{sp^{-}}{p^{-}-1},$ we have
	    \[
	        \left|\text{P.V.} \int_{B_\rho(x)}g\left(D_s u\right)
	        \frac{dy}{|x-y|^{n+s}} \right|\le C 
	        \int_{B_\rho(x)}|x-y|^{(\beta-s)p^{-}-\beta}dy = C\rho^{(\beta-s)p^{-}-\beta} \eqqcolon C_\rho\to 0
	    \]as $\rho \to 0$.

	    \medskip
	    
	    \noindent {\it Case 2:} $\nabla u(x)\neq 0.$
	    Now, proceeding as Case 2 in the proof of Lemma \ref{lemma:PV1}, we have that
        there is a constant $C$ independent of $x$ and $\rho$ such that
          \begin{align*}
	        &\left|\int_{B_{\varepsilon,\rho}}g\left(D_s u\right)
	        \frac{dy}{|x-y|^{n+s}}\right|\le\\
	        &\le C\left\{\int_{_{B_{\varepsilon,\rho}}}
	        \dfrac{H(u,L,x,y)^{p^--2}}{|x-y|^{N+sp^-}}|u(y)-L(y)| dy
	        +\int_{_{B_{\varepsilon,\rho}}}
	        \dfrac{H(u,L,x,y)^{p^+-2}}{|x-y|^{N+sp^+}}|u(y)-L(y)| dy
	        \right\},
	    \end{align*}
	    where $L(x)\coloneqq u(x)+\nabla u(x)\cdot (y-x)$ and 
	    $B_{\varepsilon,\rho}= B_\varepsilon(x)\setminus B_\rho(x)$
	    The rest of the proof follows the same line as that of \cite[Lemma 
	    3.7]{KKL}.
    \end{proof}
    
    \begin{remark}
        It is worth mentioning that, as far as we know, the notion of viscosity 
        solution is new for the case
        $1<p^- \le\tfrac{2}{2-s}.$ See, for instance, \cite{Mayte}.
    \end{remark}
    
   \begin{remark} \label{PVwelldefined}
        Lemmas \ref{lemma:PV1} and \ref{lemma:PV2} prove that
        the principal values are well-defined for functions $\psi$ that 
        are smooth enough. If additionally $\psi \in L_g(\mathbb{R}^n)$, then 
        for $x \in \mathbb{R}^n$ and $0<\rho <1$, we have
        \begin{equation*}
            \int_{\mathbb{R}^n\setminus 
            B_\rho(x)}g
            \left(\dfrac{\psi(x)-\psi(y)}{|x-y|^s} 
            \right)\dfrac{dy}{|x-y|^{n+s}} < \infty.
        \end{equation*}Indeed, take $R>0$ such that 
        $B_{\rho}(x)\subset B_R$. Then, using Lemma A.5 
        from \cite{BSV} it holds that
        $$
            |x-y|\geq \frac{1+R}{\rho}(1+|y|),\quad 
            y\in\mathbb{R}^n\setminus B_{\rho}(x).
        $$
       
        Therefore,
        \begin{equation*}
            \begin{split}
                \int_{\mathbb{R}^n\setminus 
                B_\rho(x)}&g\left(\frac{|\psi(x)-\psi(y)|}
                {|x-y|^s}\right)\frac{dy}{|x-y|^{n+s}} \\
                &\leq \left(\frac{\rho}{1+R}\right)^{n+s}\int_{\mathbb{R}^n
                \setminus B_{\rho}(x)}g\left(\left(\frac{\rho}{1+R}\right)^s
                \frac{|\psi(x)|+|\psi(y)|}{1+|y|^s}\right)\frac{dy}{1+|y|^{n+s}}\\& \leq 
                C{\left(\frac{\rho}{1+R}\right)^{n+sp^{-}}}\int_{\mathbb{R}^n
                \setminus B_{\rho}(x)}g\left(\frac{|\psi(x)|+|\psi(y)|}{1+|y|^s}\right)
                \frac{dy}{1+|y|^{n+s}}< \infty,
            \end{split}
        \end{equation*}
        {where in the last inequality we have used \eqref{gg product}, Lemma \ref{tineqg} and 
        the fact that the constant function $\psi(x)$ and $\psi$ are in $L_g(\mathbb{R}^n)$.}
        Therefore, we get that $(-\Delta_g)^s \psi$ is well-defined for the test functions 
        considered in Definition \ref{dv}.
    \end{remark}

    \medskip

    Next, we will prove that when $u$ is a viscosity solution, then it is also 
    the case for $u_\varepsilon$ with a slightly different equation.

   \begin{lemma}\label{infconvsol}
	    Let $u:\mathbb{R}^n\to\mathbb{R}$ be a bounded and lower semicontinuous function in 
	    $\mathbb{R}^n$, and $f=f(x,t,\eta)$ 
	    be continuous in 
	    $\Omega\times\mathbb{R}\times\mathbb{R}$ and non increasing in $t$. If $u$ is a viscosity 
	    supersolution of 
	    \[
	        (-\Delta_g)^s u = f(x,u, D_g^s u)\quad \text{in }\; \Omega
        \]
        then its infimal convolution $u_\varepsilon$ is a viscosity supersolution of 
        \begin{equation}\label{eqinfconv}
	        (-\Delta_g)^s u_\varepsilon 
	        = f_\varepsilon(x,u_\varepsilon,D_g^su_\varepsilon)
	        \quad \text{in }\; \Omega_{r(\varepsilon)},
	    \end{equation}
	    where $\Omega_{r(\varepsilon)}\coloneqq\left\{x\in \Omega\colon 
	    \textnormal{dist}(x,\partial\Omega)>r(\varepsilon) \right\}$ and 
    	\begin{equation}\label{fepsilon}
	        f_\varepsilon(x,t,\eta) : = \inf_{y\in 
	        B_{r(\varepsilon)}(x)}f(y,t,\eta).
	    \end{equation}
	    Moreover, 
	    \begin{equation}\label{ineqinfconv}
	        (-\Delta_g)^su_\varepsilon(x)\geq 
	        f_\varepsilon(x,u_\varepsilon(x),D_g^su_\varepsilon(x))\quad 
	        \text{a.e. }\, x\in\Omega_{r(\varepsilon)}.
	    \end{equation}
    \end{lemma}

    \begin{proof}
    	Let $z\in B_{r(\varepsilon)}(0)$ and define 
    	\[
    	    \phi_z(x)\coloneqq u(x+z)+\frac{|z|^q}{q\varepsilon^{q-1}},\quad x\in\mathbb{R}^n.
        \]
        Then we claim that $\phi_z$ is a viscosity supersolution of \eqref{eqinfconv}. Indeed, since $u$ 
        is bounded and lower semicontinuous, $\phi_z$ satisfies items (i), (ii) and (iv) from Definition \ref{dv}.

    	Now take $x_0\in \Omega_{r(\varepsilon)}$ and let 
    	$\varphi\in C^2(B_r(x_0))\cap L_g(\mathbb{R}^n)$ such that $\varphi(x_0)=\phi_z(x_0)$, 
    	$\varphi\leq \phi_z$ in $\mathbb{R}^n$ with $B_r(x_0)\subset \Omega_{r(\varepsilon)}$ and 
    	$\varphi$ satisfies (a) or (b) from Definition \ref{dv}. We put 
    	$y_0\coloneqq z+x_0$. Then $y_0\in B_{r(\varepsilon)}(x_0)\subset \Omega$ since $z\in 
    	B_{r(\varepsilon)}(0)$ and $x_0\in \Omega_{r(\varepsilon)}$. Now define 
    	$$
    	    \widetilde{\varphi}(y) \coloneqq \varphi(y-z)-\frac{|z|^q}{q\varepsilon^{q-1}}.
    	$$
    	Observe that $\varphi(\xi) = \widetilde{\varphi}(\xi+z)+\frac{|z|^q}{q\varepsilon^{q-1}}.$ Therefore
    	\begin{equation}\label{ineqphi}
    	\begin{split}
    	(-\Delta_g)^s\varphi(x_0) & = \text{P.V.}     \int_{\mathbb{R}^{n}}g\left(\frac{\varphi(x_0)
    	-\varphi(\xi)}{|x_0-\xi|^s}\right)\frac{d\xi}{|x_0-\xi|^{n+s}} \\&= 	
    	\text{P.V.} 
    	\int_{\mathbb{R}^{n}}g\left(\frac{\widetilde{\varphi}(y_0)
    	-\widetilde{\varphi}(\xi+z)}{|y_0-z-\xi|^s}\right)
    	\frac{d\xi}{|y_0-z-\xi|^{n+s}} \\& = P.V. 
    	\int_{\mathbb{R}^{n}}g\left(\frac{\widetilde{\varphi}(y_0)
    	-\widetilde{\varphi}(\xi)}{|y_0-\xi|^s}\right)
    	\frac{d\xi}{|y_0-\xi|^{n+s}} = (-\Delta_g)^s\widetilde{\varphi}(y_0).
    	\end{split}
    	\end{equation}
    	
    	On the other hand, 
    	$$
    	    \widetilde{\varphi}(y_0) = \varphi(x_0)-\frac{|z|^q}{q\varepsilon^{q-1}} 
    	    = \phi_z(x_0)-\frac{|z|^q}{q\varepsilon^{q-1}} = u(y_0)
    	 $$
    	 and
    	$$
    	    \widetilde{\varphi}(y)=\varphi(x)-\frac{|z|^q}{q\varepsilon^{q-1}}\leq 
    	    \phi_z(x)-\frac{|z|^q}{q\varepsilon^{q-1}} = u(y).
         $$ 
    	Thus, by \eqref{ineqphi} and the fact that $u$ is a viscosity supersolution 
    	of \eqref{problem} in $\Omega$ we get 
    	\begin{equation*}
    	    \begin{split}
    	        (-\Delta_g)^s\varphi(x_0) 
    	        & = (-\Delta_g)^s\widetilde{\varphi}(y_0)\geq 
    	        f(y_0,\widetilde{\varphi}(y_0),D_g^s\widetilde{\varphi}(y_0))  = 
    	        f(x_0+z,\varphi(x_0)-\tfrac{|z|^q}{q\varepsilon^{q-1}},D_g^s
    	        \widetilde{\varphi}(y_0))\\&\geq 
    	        f(x_0+z,\varphi(x_0),D_g^s\widetilde{\varphi}(y_0)),
    	    \end{split}
    	\end{equation*}where in the last inequality we used that $f$ is non 
    	increasing in the second variable.
    	
    	Now, reasoning as in \eqref{ineqphi} it holds that $D_g^s\widetilde{\varphi}(y_0) = D_g^s\varphi(x_0)$. Moreover, since $x_0+z\in B_{r(\varepsilon)}(x_0)$, the definition of $f_\varepsilon$ yields 
    	\begin{equation*}
    	(-\Delta_g)^s\varphi(x_0)  \geq f_\varepsilon(x_0,\varphi(x_0),D_g^s\varphi(x_0)).
    	\end{equation*}Hence, for every $z\in B_{r(\varepsilon)}(0)$, $\phi_z$ is a viscosity supersolution of \eqref{eqinfconv}.
    	
    	Now we check that $u_\varepsilon$ satisfies Definition \ref{dv} for the equation 
    	\eqref{eqinfconv}. Again, since $u\in L^\infty(\mathbb{R}^n)$, 
    	$u_\varepsilon$ fulfills conditions (i) and (iv). Moreover, from (iii) in 
    	Lemma \ref{propinfconv} $u_\varepsilon$ is locally Lipschitz, so assumption 
    	(ii) from Definition \ref{dv} is also satisfied. 
    
        We verify now condition (iii). Take $\psi\in C^2(B_r(x_0))\cap 
        L_g(\mathbb{R}^n)$ such that $B_r(x_0)\subset \Omega_{r(\varepsilon)}$, 
        $\psi(x_0)=u_\varepsilon(x_0)$, $\psi\leq u_\varepsilon$ in $\mathbb{R}^n$ 
        and $\psi$ satisfies (a) or (b). By (i) and (v) 
        in Lemma \ref{propinfconv} we can write
        \[
            u_\varepsilon(x) = \inf_{y\in B_{r(\varepsilon)}(x)}\left(u(y)+\frac{|x-y|^q}{q\varepsilon^{q-1}}\right) = \inf_{z\in B_{r(\varepsilon)}(0)}\phi_z(x),\;\; x\in\mathbb{R}^n,
        \]
        and there is $\overline{z}\in B_{r(\varepsilon)}(0)$ such that 
        $u_\varepsilon(x_0)=\phi_{\overline{z}}(x_0)$. Moreover, by definition 
        $\psi\leq u_\varepsilon\leq \phi_{\overline{z}}$ in $\mathbb{R}^n$. Thus we 
        can employ $\psi$ as a test function for the problem satisfied by 
        $\phi_{\overline{z}}$ and get
        \[
            (-\Delta_g)^s\psi(x_0)\geq f_\varepsilon(x_0,\psi(x_0),D_g^s\psi(x_0)).
        \]
        Therefore $u_\varepsilon$ is a viscosity supersolution of \eqref{eqinfconv}. 
    
        Now we prove \eqref{ineqinfconv}. By (iii) from Lemma \ref{propinfconv} we 
        can fix $x\in \Omega_{r(\varepsilon)}$ such that $u_\varepsilon$ is twice 
        differentiable at $x$. We first assume that $p^{-}>\frac{2}{2-s}$ or $\nabla 
        u_\varepsilon (x) \neq 0$. Take $r>0$ such that $B_r(x)\subset 
        \Omega_{r(\varepsilon)}$ and define 
        $$
            \psi_\delta(y)\coloneqq  u_\varepsilon(x)+\nabla 
            u_\varepsilon(x)(x-y)+\frac{1}{2}(D^2u_\varepsilon(x)-\delta I)(x-y)^2,
        $$
        with $\delta>0$ and $I$ the identity matrix. Notice that $\psi_\delta(x) = 
        u_\varepsilon(x)$ and $\psi_\delta\in C^2(B_r(x))$.  Consider now the 
        function 
        $$
            \psi_r(y)\coloneqq \left\{\begin{array}{ll}
            \psi_\delta(y) & \text{if }\, y\in B_r(x),\\
            u_\varepsilon(y) & \text{if }\, y\in\mathbb{R}^n\setminus B_r(x).
            \end{array} \right.
        $$
        Then, for $\delta>0$ big enough, $\psi_r\leq u_\varepsilon$. Also 
        $\psi_r\in L_g(\mathbb{R}^n)$ since $\psi_\delta$ is bounded in $B_r(x)$ 
        and $u_\varepsilon\in L_g(\mathbb{R}^n)$. Finally observe that 
        $\nabla\psi_r(x) = \nabla u_\varepsilon(x)$. Therefore we can use 
        $\psi_r$ as a test function for the problem solved 
        by $u_\varepsilon$ and get 
        $$
            (-\Delta_g)^s\psi_r(x)
            \geq f_\varepsilon(x,\psi_r(x),D_g^s\psi_r(x)).
        $$
    
        Now observe that $\psi_\delta\in C^2(\mathbb{R}^n)$, hence $\psi_\delta\in 
        C^2(B_1(x))$. Then, for $y\in B_1(x)$, 
        \[
            |\psi_\delta(x)-\psi_\delta(y)|\leq \sup_{z\in B_1(x)}|\nabla \psi_\delta(z)||x-y| = C(\psi)|x-y|. 
        \]
        Thus by  \eqref{ineqGab} we have for $0<r<1$  
        \begin{equation*}
            \begin{split}
                \int_{B_r(x)}G\left(|D_s\psi_\delta|\right)\frac{dy}{|x-y|^n} & \leq 
                \int_{B_r(x)}G(C|x-y|^{1-s})\frac{dy}{|x-y|^n} \\
                &\leq 
                G(C)\int_{B_r(x)}|x-y|^{-n+1-s}\,dy = C(n,s,\psi)r^{1-s}.
            \end{split}
        \end{equation*}
        Therefore
        \begin{equation*}
            \begin{split}
                D_g^s\psi_r(x)  & = 
                \int_{\mathbb{R}^{n}}G\left(|D_s\psi_r|\right)\frac{dy}{|x-y|^n}
                 = \int_{B_r(x)}G\left(|D_s\psi_\delta|\right)\frac{dy}{|x-y|^n} + 
                \int_{\mathbb{R}^n\setminus B_r(x)}G\left(|D_s 
                u_\varepsilon|\right)\frac{dy}{|x-y|^n}\\& = O(r^{1-s}) + 
                \int_{\mathbb{R}^{n}\setminus B_r(x)}G\left(|D_s 
                u_\varepsilon|\right)\frac{dy}{|x-y|^n}.
            \end{split}
        \end{equation*}
        Then 
        \begin{equation}\label{limDgspsi}
            \lim_{r\to 0}D_g^s\psi_r(x) = D_g^su_\varepsilon(x).
        \end{equation}
        
        On the other hand, note that $\nabla u_\varepsilon (x)\neq 0$ implies that 
        $B_r(x)\Subset \{d_{\psi_\delta}>0\}$ for $r$ small enough, since 
        $\psi_{\delta}\in C^2(\mathbb{R}^n)$ and $\nabla\psi_{\delta} (x) = \nabla 
        u_\varepsilon(x)$. 
        Hence, by Lemma \ref{lemma:PV1} 
        \begin{equation*}
            \begin{split}
                (-\Delta_g)^s\psi_r(x) & = \int_{\mathbb{R}^{n}}g\left(D_s 
                \psi_r\right)\frac{dy}{|x-y|^{n+s}} \\& = 
                \int_{\mathbb{R}^{n}\setminus B_r(x)}g\left(D_s 
                u_\varepsilon\right)\frac{dy}{|x-y|^{n+s}} 
                +\int_{B_r(x)}g\left(D_s\psi_\delta\right)\frac{dy}{|x-y|^{n+s}}\\& 
                \leq \int_{\mathbb{R}^{n}\setminus B_r(x)}g\left(D_s 
                u_\varepsilon\right)\frac{dy}{|x-y|^{n+s}}+o_r(1) 
            \end{split}
        \end{equation*}where $o_r(1)\to 0$ as $r\to 0$. Then
        \begin{equation}\label{ineq0}
            \int_{\mathbb{R}^{n}\setminus B_r(x)}g\left(D_s 
            u_\varepsilon\right)\frac{dy}{|x-y|^{n+s}}  \geq 
            (-\Delta_g)^s\psi_r(x)-o_r(1) \geq 
            f_\varepsilon(x,\psi_r(x),D_g^s\psi_r(x))-o_r(1).
        \end{equation}
        Passing to the limit as $r\to 0$ in \eqref{ineq0} 
        and using \eqref{limDgspsi} we get 
        \begin{equation*}
            (-\Delta_g)^su_\varepsilon(x)  
            = \text{P.V.} \int_{\mathbb{R}^{n}}g\left(D_s u_\varepsilon
            \right)\frac{dy}{|x-y|^{n+s}}  
            \geq f_\varepsilon(x,u_\varepsilon(x),D_g^su_\varepsilon(x)).
        \end{equation*}
    
        Now consider the case $1<p^{-}\leq \frac{2}{2-s}$ and $\nabla 
        u_\varepsilon(x) =     0$. By (vi) from Lemma \ref{propinfconv}  we have 
        \[
            u(x) = u_\varepsilon(x)\leq u(y)+\frac{|x-y|^q}{q\varepsilon^{q-1}},\;        \text{for all }\, y\in\mathbb{R}^n, \, q>\frac{sp^{-}}{p^{-}-1}\geq     2.
        \]
        Define 
        $$
            \zeta_r(y) \coloneqq 
            \left\{\begin{array}{ll}
            u(x)-\frac{|x-y|^q}{q\varepsilon^{q-1}}& y\in B_r(x),  \\
             u_\varepsilon(y)& y\in\mathbb{R}^n\setminus B_r(x) 
            \end{array}\right.
        $$
        Then $\zeta_r\in C^2_q(B_r(x))\cap L_g(\mathbb{R}^n)$ and     clearly 
        $\zeta_r(x) = u(x)$ and $\zeta_r\leq u$. Therefore we can use $\zeta_r$     
        as a test function for the problem solved by $u$ and get 
        \begin{equation}\label{ineqzetar}
            (-\Delta)_g^s\zeta_r(x)\geq f(x,\zeta_r(x),D_g^s\zeta_r(x))\geq 
            f_\varepsilon(x,u_\varepsilon(x),D_g^s\zeta_r(x)).
        \end{equation}
    
        By Lemma \ref{lemma:PV2} it holds that
        \begin{equation*}
            (-\Delta)_g^s\zeta_r(x)   \leq \int_{\mathbb{R}^n\setminus B_r(x)} 
            g\left(D_s u_\varepsilon\right)\frac{dy}{|x-y|^{n+s}} + o_r(1).
        \end{equation*}On the other hand, by \eqref{ineqGab} we have for $0<r<1$
        \begin{equation*}
            \begin{split}
                \int_{B_r(x)}G\left(|D_s\zeta_r|\right)\frac{dy}{|x-y|^n} & = 
                \int_{B_r(x)}G\left(\frac{|x-y|^q/q\varepsilon^{q-1}}{|x-y|^s}\right)
                \frac{dy}{|x-y|^n} \\& \leq  
                \int_{B_r(x)}G\left(\frac{1}{q\varepsilon^{q-1}}\right)|x-y|^{q-s-n}
                \,dy  \leq C(q,\varepsilon, n, s) r^{q-s} 
            \end{split}
        \end{equation*}
        and
        \begin{equation*}
            D_g^s\zeta_r(x)  = \int_{\mathbb{R}^n\setminus B_r(x)}G\left(|D_s 
            u_\varepsilon|\right)\frac{dy}{|x-y|^{n}} + O(r^{q-s}).
        \end{equation*}
        Thus, passing to the limit in \eqref{ineqzetar} 
        we get \eqref{ineqinfconv}.
    \end{proof}
 
 \subsection{Results on weak solutions}
    Our first result regarding weak solutions is a Caccioppoli type estimate. 
    \begin{proposition}\label{caccio} 
        Let $f \in C(\Omega \times \mathbb{R}\times \mathbb{R})$ satisfy 
        \eqref{growth f} and $u \in L^{\infty}(\mathbb{R}^{n})$ be a weak supersolution of
            \eqref{problem}. Then, there is a positive constant $C=C(p, K, \varphi)$ 
            such that
            \begin{equation}\label{cacci}
                    \int_K\int_{\mathbb{R}^{n}}G\left(|D_su| \right)G(\xi(x))d\mu 
                    \leq C\left[ G\left(osc(u)\right)\left(\int_K\int_{\mathbb{R}^{n}}
                    G\left(|D_s\xi| \right)d\mu + \gamma_{\infty, u}\right) + osc(u)\right],
            \end{equation}
            for all $\xi \in C_0^{\infty}(\Omega)$, $\xi \in [0, 1]$, 
            where $K=\textnormal{supp}(\xi),$ and 
            \[
                \gamma_{\infty,u}\coloneqq
                \max\{\gamma(t)\colon t\in[-\|u\|_{L^\infty(\mathbb{R}^n)}, \|u\|_{L^\infty(\mathbb{R}^n)}]\}.
            \]  
    \end{proposition}
    \begin{proof}
        Let $\xi \in C_0^{\infty}(\Omega)$, $\xi \in [0, 1]$, and take $K=\text{supp}(\xi)$. 
        Define
        \[
	        \varphi(x)\coloneqq	
	        \begin{cases}
	            \left(\sup\limits_{\mathbb{R}^{n}}u-u(x)\right)G(\xi(x))  
	                &\text{if } x\in \Omega,\\[7pt]
	            0\quad  &\text{if } x \in \mathbb{R}^{n}\setminus \Omega.
	        \end{cases}
	    \]
	    Observe that
	    \[
	        \varphi(x)-\varphi(y)=-(u(x)-u(y))G(\xi(x)) + 
	        (G(\xi(x))-G(\xi(y)))\left(\sup_{\mathbb{R}^{n}}u-u(y)\right).
	    \]
	    for any $x,y\in \mathbb{R}^n.$ Then, since $u$ is a weak supersolution, we have
    	\begin{equation}\label{eq 1}
        	\begin{aligned}
                &\int_\Omega f(x, u, D_g^{s}u)\varphi\,dx  \leq 	
                \int_{\mathbb{R}^{n}}\int_{\mathbb{R}^{n}} g\left(D_s u\right) 
                 D_s\varphi \, d\mu \\ 
                & = - \int_{\mathbb{R}^{n}}\int_{\mathbb{R}^{n}} 
                    g\left(D_s u\right)D_su\,G(\xi(x)) \,d\mu + \int_{\mathbb{R}^{n}}\int_{\mathbb{R}^{n}}  g\left(D_s 
                    u\right)
                D_s(G\circ\xi)\left(\sup_{\mathbb{R}^{n}}u-u(y)\right)\,d\mu.
            \end{aligned}
	    \end{equation}
	    
	    Now, since $g$ is odd, we have that $g(t)t = g(|t|)|t|.$ Hence, using the inequality  \eqref{H11}, we have
	    \begin{align*}
	       \int_{\mathbb{R}^{n}}\int_{\mathbb{R}^{n}} 
	        g\left(D_s u\right)D_s u \,G(\xi(x))
	        \,d\mu &= \int_{\mathbb{R}^{n}}\int_{\mathbb{R}^{n}} 
	        g\left(|D_s u|\right)|D_s u| G(\xi(x))
	        \,d\mu \\
	         &\geq p^{-}\int_{\mathbb{R}^{n}}\int_{\mathbb{R}^{n}} G\left(|D_s u|\right) 
	         G(\xi(x))\,d\mu.
	    \end{align*}
	    Thus, from \eqref{eq 1} it follows that
	    \begin{equation}\label{est 1}
	        \int_{\mathbb{R}^{n}}\int_{\mathbb{R}^{n}} G\left(|D_s u|\right) G(\xi(x))
	        \,d\mu \le (I)-(II),
	    \end{equation}
	    where
	    \[
	        (I)=\int_{\mathbb{R}^{n}}\int_{\mathbb{R}^{n}}  
	            g\left(D_s u\right)\,
	            D_s(G\circ\xi)\left(\sup_{\mathbb{R}^{n}}u-u(y)\right)\,d\mu,
	    \quad \text{ 
	    and }\quad
	        (II)=\int_\Omega f(x, u, D_g^{s}u)\varphi\,dx.
    	\]
	
	    We first treat the integral $(I)$. By \eqref{inq G and g} and \eqref{Young}, we have
	    \begin{equation}\label{I-1}
	        \begin{split}
	           (I) &\leq C\int_{\mathbb{R}^{n}}\int_{\mathbb{R}^{n}}  g\left(|D_s u|\right)
	                |D_s \xi|
	                (g(\xi(x))+ g(\xi(y)))\,osc(u)\,d\mu \\ 
	            & \leq C\int_{K}\int_{\mathbb{R}^{n}}  g(|D_s u|)g(\xi(x))|D_s \xi|
	                \,osc(u)\,d\mu \\ 
	            & \leq C \left[\delta \int_{K}\int_{\mathbb{R}^{n}} \tilde{G}
	                \left( g(|D_s u|)g(\xi(x))\right)d\mu +  C_\delta 
	                G\left(osc(u)\right)\int_{K}\int_{\mathbb{R}^{n}} G\left( |D_s\xi|\right) 
	                d\mu\right].
	                \end{split}
	                \end{equation}By the $\Delta'$-condition for $\tilde{G}$ and Lemma \ref{G g},
	                \begin{equation}\label{I-2}
	        \begin{split}
                  \int_{K}\int_{\mathbb{R}^{n}} \tilde{G}
	                \left( g(|D_s u|)g(\xi(x))\right)d\mu & \leq C
	                \int_{K}\int_{\mathbb{R}^{n}} \tilde{G}\left( g(|D_s 
	                u|)\right)\tilde{G}\left(g(\xi(x))\right)d\mu \\ 
	           & \leq C\int_{K}\int_{\mathbb{R}^{n}} G( |D_s u|)G(\xi(x))d\mu.  \end{split}
	    \end{equation}As a result, from \eqref{I-1} and \eqref{I-2}, there holds
	    \begin{equation}\label{I d}
	        (I) \leq C\left[ \delta\int_{K}\int_{\mathbb{R}^{n}} G( |D_s u|)G(\xi(x))d\mu +  C_\delta 
	                G\left(osc(u)\right)\int_{K}\int_{\mathbb{R}^{n}} G\left( |D_s\xi|\right) 
	                d\mu\right].
	    \end{equation}

	    Now, we estimate the term  (II) in \eqref{est 1}. By the assumption \eqref{growth f}, we get
	    \begin{equation}\label{II-1}
    	    \begin{split}
	            (II) &\leq \gamma_{\infty, u}\int_K \tilde{G}^{-1}(|D_g^{s} u|)\,osc(u)G(\xi(x))\,dx + 
	            \|\phi\|_{L^{\infty}(\Omega)}|K|osc(u)
	            \end{split}
	   \end{equation}
	    Now, by \eqref{H11} and \eqref{Young}
	            \begin{equation}\label{II-2}
    	    \begin{split}
    	   \int_K \tilde{G}^{-1}(|D_g^{s} u|)\,osc(u)G(\xi(x))\,dx  & \leq \int_K \tilde{G}^{-1}(|D_g^{s} u|)\,osc(u)\xi(x)g(\xi(x))\,dx \\
	            & \leq \delta\int_K \tilde{G}\left( 
	                \tilde{G}^{-1}(|D_g^{s}u|)g(\xi(x))\right)\,dx + 
	                C_\delta G(osc(u))|K| \\ & = \delta\int_K|D_g^{s}u|\tilde{G}(g(\xi(x)))\,dx + 
	                C_\delta G(osc(u))|K|.
	                	            \end{split}
	            \end{equation}
	            Moreover, by Lemma \ref{G g}, it follows that
	     \begin{equation}\label{II-3}
	            \int_K|D_g^{s}u|\tilde{G}(g(\xi(x))) 
	             \leq \int_K|D_g^{s}u|G(\xi(x))\,dx 
	                =\int_K\int_{\mathbb{R}^{n}}G(|D_g^{s}u|)G(\xi(x))d\mu
	    \end{equation}
	    Thus, combining  \eqref{II-1}-\eqref{II-3}, we get
	    \begin{equation}\label{II d}
	        (II) \leq \gamma_{\infty, u}\left( \delta \int_K\int_{\mathbb{R}^{n}}G(|D_g^{s}u|)G(\xi(x))d\mu + C_\delta G(osc(u))|K|\right)+ \|\phi\|_{L^{\infty}(\Omega)}|K|osc(u).
	    \end{equation}From \eqref{I d} and \eqref{II d}, and  choosing $\delta$ small enough, we derive \eqref{cacci}.
    \end{proof}
 
 The next lemma treats the convergence of the sources $f_\varepsilon$.

    \begin{lemma}\label{conv f}
        Let $u \in W^{s, G}(\Omega) \cap L^{\infty}(\Omega)$. Suppose that $f=f(x, t, \eta)$ is uniformly continuous in $\Omega \times 
        \mathbb{R}\times \mathbb{R}$, 
        Lipschitz continuous in $\eta$, and  satisfies 
        \eqref{growth f}. 
        Let $\psi \in C^{\infty}_0(\Omega)$, 
        $\psi \geq 0$ with $K=supp(\psi) \subset \Omega$. If
        \begin{equation}\label{assumption}
            \lim_{\varepsilon \to 0}
            \int_K\int_{\mathbb{R}^{n}}G\left(\frac{|u_\varepsilon(y)-u_\varepsilon
            (x)-(u(x)-u(y))|}{|x-y|^{s}}\right)d\mu =0,
        \end{equation}then 
        \begin{equation}
            \lim_{\varepsilon \to 0}\int_K f_\varepsilon (x, u_\varepsilon, 
            D_g^{s}u_\varepsilon)\psi\,dx = \int_K f(x, u, D_g^{s}u)\psi\,dx.
        \end{equation}
    \end{lemma}
    \begin{proof}
        Let $\varepsilon >0$, $\psi$ and $K$ as in the statement. By the uniformly continuity of $f$, for every $\rho >0$, 
        there exists $\delta >0$ such that
        $$
            |f(x, u_\varepsilon, D_g^{s} u_\varepsilon)-f(y, u_\varepsilon, D_g^{s} u_\varepsilon)| \leq \rho, \quad y \in B_\delta(x).
        $$
        Hence,
        \begin{equation}
            \int_K |f(x, u_\varepsilon, D_g^{s} u_\varepsilon)-f_\varepsilon(y, 
            u_\varepsilon, D_g^{s} u_\varepsilon)|\psi\,dx \leq \rho 
            \|\psi\|_{L^{\infty}(K)}|K|. 
        \end{equation}
        Since $\|u_\varepsilon\|_{L^{\infty}} \leq C$ for all $\varepsilon$, 
        it follows that 
        
        $$
            \max_{[-\|u_\varepsilon\|_{L^{\infty}},
            \|u_\varepsilon\|_{L^{\infty}}]}|\gamma(t)| 
            \leq 
            \max_{[-\|u\|_{L^{\infty}}, \|u\|_{L^{\infty}}]}|\gamma(t)|,
        $$
        and then we have
        $$
            |f(x, u_\varepsilon, D_g^{s} u)| \leq C\tilde{G}^{-1}(|D_g^{s} u|)+\varphi(x) \in L^{\tilde{G}}(K) \subset L^{1}(K)
        $$
        for a constant $C$ independent of $\varepsilon$. Then, by Lebesgue Convergence Theorem,
         \begin{equation}\label{part i}
            \lim_{\varepsilon \to 0}\int_K f(x, u_\varepsilon, D_g^{s} u)\psi\,dx
            = \int_Kf(x, u, D_g^{s} u)\psi\,dx.
        \end{equation}
        Moreover, the Lipschitz assumption of $f$ in $\eta$ gives
        \begin{equation}\label{eqq 29}
         \begin{split}
                 \int_K |f(x, u_\varepsilon, D_g^{s} u_\varepsilon)-
                 & f(x, u_\varepsilon, D_g^{s} u)|\psi\,dx   \leq  C\int_K | D_g^{s} u_\varepsilon- D_g^{s} u|\,dx \\ & \leq C\int_K\int_{\mathbb{R}^{n}}\bigg|G\left(\frac{|u_\varepsilon(x)-u_\varepsilon(y)|}{|x-y|^{s}} \right)-G\left(\frac{|u(x)-u(y)|}{|x-y|^{s}}\right) \bigg|\,d\mu \\ & \leq C\int_K\int_{\mathbb{R}^{n}}\frac{|h_\varepsilon(x, y)|}{|x-y|^{s}}g\left(\frac{|u(x)-u(y)|+ h_\varepsilon(x, y)}{|x-y|^{s}} \right)\,d\mu,
            \end{split}
        \end{equation}
        where, in the last inequality, we have used 
        $$|G(a+b)-G(b)| \leq |b|g(|a|+|b|), \quad a, b \in \mathbb{R},$$ 
        and
        \[
            h_\varepsilon(x, y) = u_\varepsilon(x)-u_\varepsilon(y)-(u(x)-u(y)),\,\,x \in K,\,y \in \mathbb{R}^{n}.
        \]
        Using Young's inequality with $\delta > 0$ and Lemma \ref{G g}, we get
        \begin{equation}\label{eqq 30}
            \begin{split}
                 \int_K\int_{\mathbb{R}^{n}}&\frac{|h_\varepsilon(x, y)|}{|x-y|^{s}}g\left(\frac{|u(x)-u(y)| + h_\varepsilon(x, y)}{|x-y|^{s}} \right)\,d\mu\\ &  \leq  C_\delta\int_K\int_\mathbb{R^n}G\left(\dfrac{|h_\varepsilon(x, y)|}{|x-y|^s}\right)\,d\mu  + \delta\int_K\int_\mathbb{R^n}G\left(\frac{|u(x)-u(y)|+ h_\varepsilon(x, y)}{|x-y|^{s}} \right)\,d\mu.   
            \end{split}
         \end{equation}
         Observe that by the inequality $G(s+t) \leq C(G(s)+G(t))$ 
         and the assumption \eqref{assumption}, the integrals 
        $$
            \int_K\int_\mathbb{R^n}G\left(\frac{|u(x)-u(y)|
            + h_\varepsilon(x, y)}{|x-y|^{s}} \right)\,d\mu
        $$
        remains uniformly bounded. Hence, taking limsup as $\varepsilon\to 0$ in 
        \eqref{eqq 30}, using \eqref{assumption}, and then $\delta \to 0$, we get
        from \eqref{eqq 29} that
        \begin{equation}\label{part ii}
            \int_K |f(x, u_\varepsilon, D_g^{s} u_\varepsilon)-f(x, 
            u_\varepsilon, D_g^{s} u)|\psi\,dx \to 0 \quad 
            \text{ as }\varepsilon\to 0.
        \end{equation}
        Then, combining \eqref{part i} and \eqref{part ii}, we have
        \begin{equation*}
            \begin{split}
                &  \lim_{\varepsilon \to 0}\int_K |f_\varepsilon(x, u_\varepsilon,   
                D_g^{s}u_\varepsilon)-f(x, u, D_g^{s} u)|\psi\,dx \\ & \leq     
                \lim_{\varepsilon \to 0}\int_K|f_\varepsilon(x, u_\varepsilon, 
                D_g^{s}u_\varepsilon)-f(x, u_\varepsilon, D_g^{s} u)|\psi\,dx  +  
                \int_K|f(x, 
                u_\varepsilon, D_g^{s}u)-f(x, u_\varepsilon, D_g^{s} u)|\psi\,dx  =0.
            \end{split}
    \end{equation*}This concludes the proof. 
\end{proof}

In the next two results we will study the relation between the weak and the pointwise formulation of solutions. We distinguish two cases:
$p^{-}> \tfrac{2}{2-s}$ and $1 < p^{-} \leq  \tfrac{2}{2-s}.$
 
\begin{lemma}\label{lemma 1 weak point} 
    Assume $p^{-}> \frac{2}{2-s}$ and let $u \in L_g(\mathbb{R}^{n})\cap 
    L^{\infty}(\mathbb{R}^{n})$. Then, for all $\psi \in 
    C_0^{\infty}(\Omega_{r(\varepsilon)})$, $\psi \geq 0$, we have
    \begin{equation}
        \int\int_{Q_K}g\left(D_s u_\varepsilon\right) D_s\psi\,d\mu \geq \int_K\left( 
        (-\Delta_g)^{s}u_\varepsilon \right)\psi\,dx
    \end{equation}
    where $u_\varepsilon$ is the infimal convolution of $u,$ $K$ is the support of $\psi,$
    and
    \[
        Q_K\coloneqq(K\times K) \cup \left[(\mathbb{R}^{n}\setminus K) \times K \right] \cup \left[ K \times (\mathbb{R}^{n}\setminus K)\right].
    \]
\end{lemma}
\begin{proof}
    Let $u_{\varepsilon, \delta}$ be smooth and semiconcave functions converging to $u_{\varepsilon}$  given by
    $$u_{\varepsilon, \delta}\coloneqq(u_\varepsilon * \eta_\delta)\chi_{\Omega_{r(\varepsilon)}} + u_\varepsilon \chi_{\mathbb{R}^{n}\setminus\Omega_{r(\varepsilon)}},$$where $\eta_\delta$ is the standard mollifier with support in $B_\delta$. Oberserve that $u_{\varepsilon, \delta} \in C^{2}(\Omega_{r(\varepsilon)}) \cap L_g^{s}(\mathbb{R}^{n})$. Moreover, for any $\psi \in C_0^{\infty}(\Omega_{r(\varepsilon)})$, with $K=\text{supp}\,\psi$, we have recalling that $g$ is odd that
    \begin{equation}
    \begin{split}
        &\int_{\mathbb{R}^n}\int_{\mathbb{R}^n}g\left(D_s u_{\varepsilon, \delta}\right)\dfrac{\psi(x)-\psi(y)}{|x-y|^{s+n}}\,dx\,dy=\\ &=\int_{\mathbb{R}^n}\int_{\mathbb{R}^n}
        \dfrac{g\left(D_s u_{\varepsilon, \delta}\right)\psi(x)}{|x-y|^{n+s}}\,dx\,dy  -\int_{\mathbb{R}^n}\int_{\mathbb{R}^n}
        \dfrac{g\left(D_s u_{\varepsilon, \delta}\right)\psi(y)}{|x-y|^{n+s}}\,dx\,dy  = \int_K\left((-\Delta_g)^{s}u_{\varepsilon, \delta} \right)\psi\,dx.
        \end{split}
    \end{equation}We shall take the limit as $\delta \to 0$. First, we show that
    \begin{equation}\label{conv F}
    \begin{split}
       & \lim_{\delta \to 0} \int\int_{Q_K}g\left(D_s u_{\varepsilon, \delta}\right)\dfrac{\psi(x)-\psi(y)}{|x-y|^{s+n}}\,dx\,dy  =\int\int_{Q_K}g\left(D_s u_{\varepsilon }\right)\dfrac{\psi(x)-\psi(y)}{|x-y|^{s+n}}\,dx\,dy.
        \end{split} 
    \end{equation} Let
    $$
        F_\delta(x, y)\coloneqq g\left(D_s u_{\varepsilon, 
        \delta}\right)\dfrac{\psi(x)-\psi(y)}{|x-y|^{s+n}}.
    $$
    By symmetry,
    $$
        \int\int_{Q_K}F_\delta(x, y)\,dxdy = \left(2\int\int_{K \times 
        (\mathbb{R}^{n}\setminus K)} + \int\int_{K\times K} \right)F_\delta(x, y)\,dxdy.
    $$
    Hence, in order to get \eqref{conv F}, it is enough to prove that
    $$
        |F_\delta(x, y)| \leq F(x, y)
    $$
    for any $\delta > 0$ and  some $F \in L^{1}(K \times \mathbb{R}^{n})$. 
    Let us consider
    $$
        r\coloneqq \text{dist }(K, \partial \Omega_{r(\varepsilon)}) \qquad \text{and }
        \qquad K_{r/2}\coloneqq\left\lbrace y \in 
        \Omega_{r(\varepsilon)}\colon \text{dist }(y, K) \leq \frac{r}{2}\right\rbrace.
    $$ 
    By Young's inequality with $\delta=1$,  and Lemma \ref{G g}, we have for $(x, y) \in K \times K_{r/2}$ 
    \begin{equation*}
        \begin{split}
            |F_\delta(x, y)|& \leq g(|D_s u_{\delta, \varepsilon}|)\frac{\psi(x)-\psi(y)}{|x-y|^{s+n}}  \leq  \frac{C}{|x-y|^{n}}\left(\tilde{G}\left( g(|D_s u_{\delta, \varepsilon}|)\right)+ G\left(|D_s \psi|\right)\right) \\ &  \leq \frac{C}{|x-y|^{n}}\left(G\left(|D_s u_{\delta, \varepsilon}|\right)+ G\left(|D_s \psi|\right) \right) \\ & \leq C\left( |x-y|^{(1-s)p^{+}-n}+ |x-y|^{(1-s)p^{-}-n}\right) \in L^{1}(K \times K_{r/2})
        \end{split}
    \end{equation*}with a constant $C$ independent of $\delta$. Next, we consider the case 
    $(x, y)\in K \times (\mathbb{R}^{n}\setminus K_{r/2})$. Since $\text{dist }(K, \mathbb{R}^{n}\setminus 
    K_{r/2}) > 0$ and $K$ is bounded, by \cite[Lemma A.5 ]{BSV} there is $C > 0$ such that
    $$|x-y| \geq C(1+ |y|).$$Hence, for any $\delta > 0$,
    \begin{equation}\label{F F}
        |F_\delta(x, y)| \leq \dfrac{C}{1 + |y|^{n+s}}g\left( \dfrac{1 + |u_\varepsilon(y)|}{1+ |y|^{s}}\right)
        \leq 
        \dfrac{C}{1 + |y|^{n+s}}\left[g \left(\dfrac{1}{1+ |y|^{s}}\right)+ g\left( \dfrac{|u_\varepsilon(y)|}{1+|y|^{s}}\right)  \right],
    \end{equation}and thus, since $u_\varepsilon \in L_g(\mathbb{R}^{n})$, we get that the right-hand side of \eqref{F F} belongs to $L^1(K \times (\mathbb{R}^{n}\setminus K_{r/2}))$. As a result, by dominated convergence Theorem, \eqref{conv F} holds. 
    
    The next step is to prove that 
    \begin{equation}\label{fatou}
        \liminf_{\delta \to 0}\int_K (-\Delta_g)^{s}u_{\delta, \varepsilon}\psi \,dx \geq \int_K \liminf_{\delta \to 0}(-\Delta_g)^{s}u_{\delta, \varepsilon}\psi \,dx,
    \end{equation}by appealing to Fatou Lemma. Hence, it is enough to find $C$ so that
    \begin{equation}
        (-\Delta_g)^{s}u_{\delta, \varepsilon}(x) \geq -C,
    \end{equation}for all  $x \in K$ and all $\delta$. So let $r > 0$ so that $r < \text{dist }(K, \partial \Omega_{r(\varepsilon)})$. Then $B_r(x) \subset  \Omega_{r(\varepsilon)}$ for all $x \in K$, and for $y \in \mathbb{R}^{n} \setminus B_r(x)$, there holds $|x-y| > r$. Now, again by \cite[Lemma A. 5]{BSV}
    \begin{equation}
        \bigg| \int_{\mathbb{R}^{n} \setminus B_r(x)} g\left(D_s u_{\delta, \varepsilon}\right)\frac{dy}{|x-y|^{n+s}} \bigg| \leq C\int_{\mathbb{R}^{n} \setminus B_r(x)}g\left(\dfrac{1 + |u_{ \varepsilon}(y)|}{1+ |y|^{s}} \right)\dfrac{dy}{1+|y|^{n+s}},
    \end{equation}and the latter integral is uniformly bounded in $x$ since $u_\varepsilon \in L_g(\mathbb{R}^{n})$. Let us now estimate the integral over the ball $B_r(x)$. By Lemma \ref{propinfconv}, there is $C > 0$ such that
     \begin{equation}\label{Hessian}
         D^{2}u_\varepsilon (x) \leq C I, \quad a.e. \,\,x\in \Omega_{r(\varepsilon)},
     \end{equation}Hence, by symmetry and the fact that $g$ is odd, we have
    \begin{equation}\label{I r}
        \begin{split}
        I_{B_r(x)}& \coloneqq  P.\,V. \int_{B_r(x)}g\left( D_s 
            u_{\varepsilon, \delta}\right)\dfrac{dx\,dy}{|x-y|^{n+s}} \\ & 
            =\int_{B_r(x)}\left[g\left( D_s u_{\varepsilon, 
            \delta}\right)-g\left( \dfrac{-\nabla u_{\varepsilon, 
            \delta}(x)(y-x)}{|x-y|^{s}}\right) 
            \right]\dfrac{dx\,dy}{|x-y|^{n+s}}. 
        \end{split}
    \end{equation}
    Since for all $a, b \in \mathbb{R}$,
    $$
        g(b)-g(a)=(a-b)\int_0^{1}g'(ta+(1-t)b)\,dt,
    $$ putting
    $a=D_s u_{\varepsilon, \delta},$ and 
    $b=\tfrac{-\nabla u_{\varepsilon, \delta}(x)(y-x)}{|x-y|^{s}},$
    we have
    \begin{equation}\label{ineq for I a}
    \begin{split}
        I_{B_r(x)}& = \int_{B_r(x)}\dfrac{u_{\varepsilon, 
        \delta}(x)-u_{\varepsilon, \delta}(y)+ \nabla u_{\varepsilon, 
        \delta}(x)(y-x)}{|x-y|^{n+2s}}\left(\int_0^{1}g'(ta+(1-t)b)\,dt
        \right) dy\\ 
        & = \int_{B_r(x)}\dfrac{-D^{2}u_{\varepsilon, 
        \delta}(z)(x-y)^{2}}{|x-y|^{n+2s}}\left(\int_0^{1}g'(ta+(1-t)b)\,dt
        \right)\,dy \\ 
        & \geq \int_{B_r(x)^{+}}\dfrac{-D^{2}u_{\varepsilon, 
        \delta}(z)(x-y)^{2}}{|x-y|^{n+2s}}\left(\int_0^{1}g'(ta+(1-t)b)\,dt
        \right)\,dy,
    \end{split}
    \end{equation}
    where $z \in B_r(x)$ and $B_r(x)^{+}\coloneqq \left\lbrace y \in B_r(x)\colon D^{2}u_{\varepsilon, \delta}(y) \geq 0 \right\rbrace.$
    Next, assume that $p^{-} \geq 2$. Then, by \eqref{H1} and \eqref{H111}, we get
    \begin{equation*}
        \begin{split}
            \bigg\vert \int_0^{1}g'(ta+(1-t)b)\,dt \bigg\vert &\leq 
            C\int_0^{1}\bigg\vert \dfrac{g(ta+(1-t)b)}{ta+(1-t)b}
            \bigg\vert \,dt\\ 
            & \leq C\left(\int_0^{1}|ta+(1-t)b|^{p^{+}-2}\,dt + 
            \int_0^{1}|ta+(1-t)b|^{p^{-}-2}\,dt\right).
        \end{split}
    \end{equation*}
    By Lemma 2.4 in \cite{BM} and since $\|\nabla u_{\varepsilon, 
    \delta}\|_{L^{\infty}(K)} \leq C$ independently of $\delta$, it follows that 
    \begin{equation}\label{ineq for I}
        \begin{split}
             \int_0^{1}|ta+(1-t)b|^{p^{+}-2}\,&dt  + \int_0^{1}|ta+(1-t)b|^{p^{-}-2}\,dt\\  
            &\leq C\left( |a|^{p^{+}-2}+|b|^{p^{+}-2}+|a|^{p^{-}-2}+|b|^{p^{-}-2}\right) \\ 
            &  \leq C \left( |x-y|^{(1-s)(p^{+}-2)}+ |x-y|^{(1-s)(p^{-}-2)}\right) 
            \quad (y \in B_r(x)).
\end{split}
\end{equation}Plugging \eqref{ineq for I} into \eqref{ineq for I a}, we get 
$$I_{B_r(x)} \geq -C\int_{B_r(x)}\left( |x-y|^{(1-s)(p^{+}-2)+2-n-2s}+ |x-y|^{(1-s)(p^{-}-2)+2-n-2s}\right)\,dy \geq -C,$$uniformly in $\delta$.

     Next, assume that $\tfrac2{2-s} <  p^{-} < 2$.  Then, 
     we have again from Lemma 2.4 in \cite{BM}  that
    \begin{equation*}
        \begin{split}
            0 & \leq \int_0^{1}|g'(at+(1-t)b)|\,dt \\ 
            & \leq C\left(\int_0^{1}|at+(1-t)b|^{p^{-}-2}\,dt + 
            \int_0^{1}|at+(1-t)b|^{p^{+}-p^{-}+p^{-}-2}\,dt\right) \\
            & \leq C
            \left(\int_0^{1}|at+(1-t)b|^{p^{-}-2}\,dt + 
            \int_0^{1}
            \left(|a|^{p^{+}-p^{-}}+|b|^{p^{+}-p^{-}}\right)|at+(1-t)b|^{p^{-}-2}\,dt
            \right) \\
            & \leq C\left(|a|^{p^{+}-p^{-}}+|b|^{p^{+}-p^{-}}+1\right)|a-b|^{p^{-}-2}.
        \end{split}
    \end{equation*}
    Thus, by \eqref{Hessian}, 
    and the assumption $p^{-}> \tfrac2{2-s}$, it follows
    \[
        I_{B_r(x)} 
        \geq -C\int_{B_r(x)} 
            |x-y|^{(2-s)(p^{-}-2)+2-n-2s}
            \,dy \geq -C,
    \]
    with $C$ independent of $\delta$. 
    
    In any case, we have
    $$(-\Delta_g)^{s}u_{\varepsilon, \delta} \geq -C \text{ in }K.$$Hence, applying Fatou's Lemma, we obtain \eqref{fatou}.

    Finally, we will show that
    \begin{equation}\label{delta a cero}
        \int_K\lim_{\delta \to 0}(-\Delta_g)^{s}u_{\varepsilon, \delta}\,\psi\,dx \geq \int_K 
        (-\Delta_g)^{s}u_{\varepsilon}\,\psi\,dx.  
    \end{equation}
    Write
    \begin{equation*}
        \begin{split}
        (-\Delta_g)^{s}u_{\varepsilon, \delta}(x)& =\int_{\mathbb{R}^{n}\setminus B_r(x)}g
        \left(D_s u_{\varepsilon, \delta}\right)\dfrac{dy}{|x-y|^{n+s}}\\ 
        &+  \int_{B_r(x)}\left[g\left( D_s u_{\varepsilon, \delta}\right)-
            g\left( \dfrac{-\nabla 
            u_{\varepsilon, \delta}(y-x)}{|x-y|^{s}}\right) \right]\dfrac{dy}{|x-y|^{n+s}}.
        \end{split}
    \end{equation*}
    Then, we first have in $\mathbb{R}^{n}\setminus B_r(x)$ that
    \begin{equation*}
        g\left( D_s u_{\varepsilon, \delta}\right)\dfrac{1}{|x-y|^{n+s}} 
        \geq -Cg\left(\dfrac{1+ |u_{\varepsilon}(y)|}{|1+ |y|^{s}|} \right)
        \dfrac{1}{1+ |y|^{n+s}}.
    \end{equation*}
    Now, in the ball $B_r(x)$, 
    \begin{equation*}
        \begin{split}
            \left[g\left( D_s u_{\varepsilon, \delta}\right)
            -g\left( \dfrac{-\nabla u_{\varepsilon, \delta}(y-x)}{|x-y|^{s}}\right) 
            \right]\dfrac{1}{|x-y|^{n+s}} \geq -C\mathcal{F}(x, y),
        \end{split}
    \end{equation*}
    where
    
    \[
        \mathcal{F}(x,y)\coloneqq
        \begin{cases}
            |x-y|^{(1-s)(p^{+}-2)+2-n-2s}+ |x-y|^{(1-s)(p^{-}-2)+2-n-2s}& 
                \text{if } p^{-}\geq 2,\\
            |x-y|^{(2-s)(p^{-}-2)+2-n-2s}& 
                \text{if } p^{-}< 2.
        \end{cases}
    \]
    
    Since 
    \[
        g\left(\dfrac{1+ |u_{\varepsilon}(y)|}{|1+ |y|^{s}|} \right)\dfrac{1}{1+ |y|^{n+s}} \in L^{1}(\mathbb{R}^{n}\setminus B_r(x))  \quad \text{and }\quad \mathcal{F}(x, \cdot )\in L^{1}(B_r(x)),
    \]
    by Fatou's Lemma, we have that \eqref{delta a cero} holds. 
    This ends the proof of the lemma.
\end{proof}
    Now, we state the counterpart of Lemma \ref{lemma 1 weak point} for the range 
    $1 < p^{-} \leq \frac{2}{2-s}$.

\begin{lemma}\label{lemma 2 weak point} 
    Assume $1 < p^{-} \leq \frac{2}{2-s}$ and let $u \in L_g(\mathbb{R}^{n})\cap 
    L^{\infty}(\mathbb{R}^{n})$. Then, for all $\psi \in 
    C_0^{\infty}(\Omega_{r(\varepsilon)})$, $\psi \geq 0$, we have
    \begin{equation}
        \int\int_{Q_K}g\left(D_s u_\varepsilon\right) D_s\psi \,d\mu
        \geq \int_K\left( 
        (-\Delta_g)^{s}u_\varepsilon \right)\psi\,dx
    \end{equation}
    where $u_\varepsilon$ is the infimal convolution of $u,$  $K$ is the support of $\psi,$
    and
    \[
        Q_K=(K\times K) \cup \left[(\mathbb{R}^{n}\setminus K) \times K \right] \cup \left[ K \cup (\mathbb{R}^{n}\setminus K)\right].
    \]  
\end{lemma}
\begin{proof}
    For $\rho > 0$ and non negative $\psi \in C_0^{\infty}(\Omega)$, we have
    \begin{equation}\label{eqq 0}
        \begin{split}
            \int\int_{Q_K}
            g\left(\dfrac{u_\varepsilon(x)-u_\varepsilon(y)}{(|x-y|+\rho)^{s}}\right)
            D_s\psi\,
            d\mu \geq \int_K\int_{\mathbb{R}^{n}}
            g\left(\dfrac{u_\varepsilon(x)-u_\varepsilon(y)}{(|x-y|+\rho)^{s}}\right) 
            \dfrac{\psi(x)}{|x-y|^{n+s}}\,dxdy,
        \end{split}
    \end{equation}
    since
    \[
        Q_K=(K\times K) \cup \left[(\mathbb{R}^{n}\setminus K) \times K \right] \cup \left[ K 
        \times (\mathbb{R}^{n}\setminus K)\right].
    \]
    Now, for $x, y \in K$, it holds since $u_\varepsilon$ is locally Lipschitz, that
    \begin{equation}\label{eqq 1}
        \begin{split}
            g\left(\dfrac{u_\varepsilon(x)-u_\varepsilon(y)}{(|x-y|+\rho)^{s}}\right)
            \dfrac{\psi(x)-\psi(y)}{|x-y|^{n+s}} & \leq 
            C\left(\dfrac{|u_\varepsilon(x)-u_\varepsilon(y)|^{p^{+}-1}}{|x-y|^{n+sp^{+}}} 
            +\dfrac{|u_\varepsilon(x)-u_\varepsilon(y)|^{p^{-}-1}}{|x-y|^{n+sp^{-}}} \right) \\ 
            &  \leq C\left(|x-y|^{(1-s)p^{+}-n}+|x-y|^{(1-s)p^{-}-n}\right).
        \end{split}
    \end{equation}
    On the other hand, if $x\in K$ and $y \in \mathbb{R}^{n}\setminus K$, recalling that supp 
    $\psi \subset \subset K$,  we have
    \begin{equation}\label{eqq 2}
        \bigg\vert g\left( \dfrac{u_\varepsilon(x)-u_\varepsilon(y)}{(|x-y|+\rho)^{s}}\right)  
        \bigg\vert
        \dfrac{|\psi(x)-\psi(y)|}{|x-y|^{s}} \leq Cg\left( \dfrac{1+ 
        |u_\varepsilon(y)|}{1+ |y|^{s}}\right)\dfrac{1}{1+|y|^{s}}.
    \end{equation}
    Hence, from \eqref{eqq 1} and \eqref{eqq 2}, 
    we may apply Lebesgue theorem to let $\rho \to 0$ in \eqref{eqq 0} and get
    \begin{equation}\label{eqq 3}
        \begin{split}
            \lim_{\rho \to0}
            \int\int_{Q_K}g\left(\dfrac{u_\varepsilon(x)-u_\varepsilon(y)}{(|x-y|+\rho)^{s}}
            \right)D_s\psi\, d\mu= 
            \int\int_{Q_K}g\left(D_su_\varepsilon\right)D_s\psi\, d\mu.
    \end{split}
\end{equation}To treat the right-hand side of \eqref{eqq 0}, we introduce the function
$$\mathcal{F}_\rho(x):= \int_{\mathbb{R}^{n}}g\left(\dfrac{u_\varepsilon(x)-u_\varepsilon(y)}{(|x-y|+\rho)^{s}}\right) \frac{1}{|x-y|^{n+s}}\,dy, \quad x \in K.$$We will prove that there is $G \in L^{1}(K)$ such that
\begin{equation}\label{eqq 4}
\mathcal{F}_\rho(x)\psi(x)\geq - G(x), 
\end{equation}in $K$. Hence, we will apply Fatou's Lemma to $\mathcal{F}_\rho(x)\psi(x)+ G(x)$ to get
\begin{equation}\label{eqq 5}
\liminf_{\rho \to 0}\int_K \mathcal{F}_\rho(x)\psi(x)\,dx \geq \int_K \liminf_{\rho \to 0} \mathcal{F}_\rho(x)\psi(x)\,dx.
\end{equation}In order to prove \eqref{eqq 4}, let 
 $x \in K$ and choose $\eta > 0$ such that $B_\eta(x) \subset \Omega_{r(\varepsilon)}$. By Lemma A. 5 in \cite{BSV}, for any $y \in \mathbb{R}^{n}\setminus B_\eta(x)$
$$|x- y| \geq \left(\dfrac{\eta}{1+\eta+R}\right)(1+|y|) $$where $R > 0$ satisfies that the ball $B_R(0)$ contains $K$. Hence, there is a constant $C > 0$ independent of $x$ and $\rho$ such that
    \begin{equation}\label{eqq 6}
        \bigg\vert \int_{\mathbb{R}^{n}\setminus 
        B_\eta(x)}g\left(\dfrac{u_\varepsilon(x)-u_\varepsilon(y)}{(|x-y|+\rho)^{s}}\right) 
        \frac{dy}{|x-y|^{n+s}}\bigg\vert \leq C\int_{\mathbb{R}^{n}\setminus 
        B_\eta(x)}
        g\left( \dfrac{1+|u_\varepsilon(y)|}{1+|y|^{s}}\right)\dfrac{dy}{1+|y|^{s+n}}. 
    \end{equation}
    Now, we study the  integrals in the ball $B_\eta(x)$. By Lemma \ref{propinfconv}, for each 
    $x \in \Omega_{r(\varepsilon)}$, there is $\hat{x} \in B_{r(\varepsilon)}(x)$ such that 
    \begin{equation*}
        u_\varepsilon(x)= \varphi_{\hat{x}}(x),
    \end{equation*}where
    $$
        \varphi_{z}(y):= u(z)+ \dfrac{|y-z|^{q}}{q\varepsilon^{q-1}}, \,\, y, z \in \mathbb{R}^{n}.
    $$
    with $q > \tfrac{sp^{-}}{(p^{-}-1)}$. Then
    \begin{equation*}
        u_\varepsilon(x)-u_\varepsilon(y) = \varphi_{\hat{x}}(x)-\inf_{z \in 
        \mathbb{R}^{n}}\varphi_z(y)  \geq  \varphi_{\hat{x}}(x) - \varphi_{\hat{x}}(y).
    \end{equation*}Since $g$ is non-decreasing,
    $$
        g\left(\dfrac{u_\varepsilon(x)-u_\varepsilon(y)}{(|x-y|+\rho)^{s}} \right) \geq g\left(\dfrac{ \varphi_{\hat{x}}(x) - \varphi_{\hat{x}}(y)}{(|x-y|+\rho)^{s}} \right).
    $$
    Hence
    \begin{equation}\label{eqq 7}
        \int_{B_\eta(x)}g\left(
        \dfrac{u_\varepsilon(x)-u_\varepsilon(y)}{(|x-y|
        +\rho)^{s}} \right) \dfrac{dy}{|x-y|^{s+n}} \geq  
        \int_{B_\eta(x)}g\left(\dfrac{\varphi_{\hat{x}}(x) - 
        \varphi_{\hat{x}}(y)}{(|x-y|+\rho)^{s}} \right)  
        \dfrac{dy}{|x-y|^{s+n}}. 
    \end{equation}
    Observe that by \eqref{eqq 6} and by  \eqref{eqq 7}, the inequality 
    \eqref{eqq 4} is stated if we additionally prove that 
    \begin{equation}\label{eqq 8}
        \bigg\vert \int_{B_\eta(x)}g\left(\dfrac{\varphi_{\hat{x}}(x) - 
        \varphi_{\hat{x}}(y)}{(|x-y|+\rho)^{s}} \right)  
        \dfrac{dy}{|x-y|^{s+n}}\bigg\vert \leq C.
    \end{equation}

    Let $\eta_0 \leq \eta$ small. 
    First, suppose that $x \notin B_{\eta_0}(\hat{x})$. 
    By Lemma \ref{propinfconv}
    \begin{equation}\label{eqq 9}
        \|\varphi_{\hat{x}}\|_{L^{\infty}(\Omega)} \leq C, \quad |\nabla 
        \varphi_{\hat{x}}(y)| = \dfrac{|\hat{x}-y|^{q-1}}{\varepsilon^{q-1}}  
        \leq C,
    \end{equation}
    and
    \begin{equation}\label{eqq 10}
        -CI \leq -\dfrac{q-1}{\varepsilon^{q-1}}|\hat{x}-y|^{q-2}I 
        \leq D^{2}\varphi_{\hat{x}}(y) \leq 
        \dfrac{q-1}{\varepsilon^{q-1}}|\hat{x}-y|^{q-2}I  \le {C(\varepsilon)}I.
    \end{equation}
    Observe that $|\nabla \varphi_{\hat{x}}(x)|\neq 0$ since $x \neq \hat{x}$.
    Now,  since $p^-<2$,
    we have reasoning as in the proof of \cite[Lemma 3.5]{BM} and letting $L(y)= \varphi_{\hat{x}}(x)+ \nabla \varphi_{\hat{x}}(x)(y-x)$, that 
    \begin{equation}\label{eqq 50}
        \begin{split}
            & \bigg\vert  \int_{B_\eta(x)} g\left(\dfrac{\varphi_{\hat{x}}(x) 
            - \varphi_{\hat{x}}(y)}{(|x-y|+\rho)^{s}} \right)  
            \dfrac{dy}{|x-y|^{s+n}}\bigg\vert  \\ & \qquad \leq C  \bigg\vert 
            \int_{B_\eta(x)} g\left(\dfrac{-\nabla 
            \varphi_{\hat{x}}(x)(y-x)-(y-x)^{T}D^{2}\varphi_{\hat{x}}(z)(y-x)}
            {(|x-y|+\rho)^{s}} \right)  \dfrac{dy}{|x-y|^{s+n}}\bigg\vert 
            \quad (\text{for }z \in B_\eta(x)) \\ & \qquad\leq 
            C\int_{B_\eta(x)} \bigg\vert g\left(\dfrac{-\nabla 
            \varphi_{\hat{x}}(x)(y-x)-(y-x)^{T}D^{2}\varphi_{\hat{x}}(z)(y-x)}
            {(|x-y|+\rho)^{s}} \right)  - 
            g\left(\dfrac{L(x)-L(y)}{(|x-y|+\rho)^{s}}\right) 
            \bigg\vert\dfrac{dy}{|x-y|^{s+n}} \\ & \qquad \leq C 
            \bigg[\int_{B_\eta(x)} \dfrac{\left(|\nabla 
            \varphi_{\hat{x}}(x)(y-x)|+ 
            |D^{2}\varphi_{\hat{x}}(z)||x-y|^{2}\right)^{p^{+}-2}|D^{2}
            \varphi_{\hat{x}}(z)||x-y|^{2}}{|x-y|^{n+sp^{+}}}\,dy  \\& 
            \qquad\quad  \qquad + \int_{B_\eta(x)} \dfrac{\left(|\nabla 
            \varphi_{\hat{x}}(x)(y-x)|+ 
            |D^{2}\varphi_{\hat{x}}(z)||x-y|^{2}\right)^{p^{-}-2}|D^{2}
            \varphi_{\hat{x}}(z)||x-y|^{2}}{|x-y|^{n+sp^{-}}}\,dy \bigg],
        \end{split}
    \end{equation}
    where in the last inequality we have used Lemma \ref{lema:aux1}. 
    Next, observe that 
     since
        $p^+\ge p^-$ we have
        \begin{align*}
            H(x,y,p^+)\coloneqq&\dfrac{\left(|\nabla 
            \varphi_{\hat{x}}(x)(y-x)|+ |D^{2}\varphi_{\hat{x}}(z)||x-y|^{2}\right)^{p^{+}-2}|D^{2}
            \varphi_{\hat{x}}(z)||x-y|^{2}}{|x-y|^{n+sp^{+}}}\\
            &=\dfrac{\left(|\nabla \varphi_{\hat{x}}(x)(y-x)|+ 
            |D^{2}\varphi_{\hat{x}}(z)||x-y|^{2}\right)^{p^{+}-p^{-}}}{|x-y
            |^{s(p^{+}-p^{-})}}H(x,y,p^-)\\
            & \le{C(\varepsilon)}\left(|y-x|^{(1-s)(p^{+}-p^-)}+ 
            |y-x|^{(2-s)(p^{+}-p^- )}\right)H(x,y,p^-)\\
            &\le {C(\varepsilon)} H(x,y,p^-)
        \end{align*}
        for any $y\in B_\eta(x).$ Then
        \begin{equation}\label{eqq 502}
            \bigg\vert  \int_{B_\eta(x)} g\left(\dfrac{\varphi_{\hat{x}}(x) 
            - \varphi_{\hat{x}}(y)}{(|x-y|+\rho)^{s}} \right)  
            \dfrac{dy}{|x-y|^{s+n}}\bigg\vert \leq{C(\varepsilon)}
              \int_{B_\eta(x)}
            H(x,y,p^-)\,dy 
    \end{equation}
    for $z\in B_\eta(x).$ Thus, taking  $\tau_\infty{(\varepsilon)}\coloneqq\sup_{\Omega_{r(\varepsilon)}}|D^{2}
            \varphi_{\hat{x}}|,$ we have
    \begin{equation}\label{eqq 51}
    \begin{split}
            & \int_{B_\eta(x)} 
            H(x,y,p^-)
            \,dy 
            \leq 
            \tau_\infty{(\varepsilon)}\int_{B_\eta(x)} 
            \dfrac{\left[|\nabla \varphi_{\hat{x}}(x)||x-y|+ 
            \tau_\infty|y-x|^{2}\right]^{p^{-}-2}}{|x-y|^{n+sp^{-}}} 
            |x-y|^{2}\,dy, \ \\ 
            &  \leq C\tau_\infty{(\varepsilon)}\int_0^{\eta}\left(1+ 
            \dfrac{r}{|\nabla \varphi_{\hat{x}}(x)|} \right)^{p^{-}-2}|\nabla 
            \varphi_{\hat{x}}(x)|r^{p^{-}(1-s)}\frac{dr}{r} \leq {C(\varepsilon)}.
        \end{split}
    \end{equation}


    Assume now that $x \in B_{\eta_0}(\hat{x})$. 
    Since $p^{-} < 2$, we proceed  following the proof of  \cite[Lemma 3.7]{KKL}. 
    Notice first that
    $$
        \sup_{y \in B_r(x)}|D^2 \varphi_{\hat{x}}(y)| \leq \sup_{y \in B_r(x)}C|y-\hat{x}|^{q-2} \leq C(r+|x-\hat{x}|) ^{q-2},
    $$
    for a constant $C$ depending on $q$ and $\varepsilon$. Working as in \eqref{eqq 50}, we have 
    \[
    \begin{split}
    &  \int_{B_\eta(x)} \dfrac{\left(|\nabla \varphi_{\hat{x}}(x)(y-x)|+ 
    |D^{2}\varphi_{\hat{x}}(z)||x-y|^{2}\right)^{p^{-}-2}|D^{2}
    \varphi_{\hat{x}}(z)||x-y|^{2}}{|x-y|^{n+sp^{-}}}\,dy   \\ & \quad  \leq 
    \int_0^\eta  \left(1+\dfrac{(r+|x-\hat{x}|) ^{q-2}r}{|\nabla 
    \varphi_{\hat{x}}(x)|}\right)^{p^{-}-2}(r+|x-\hat{x}|) ^{q-2} |\nabla 
    \varphi_{\hat{x}}(x)|^{p^{+}-2}r^{p^{-}(1-s)-1}\,dr  \\ 
    & \quad\leq 
    C\bigg(\int_0^{|x-\hat{x}|}  \left(1+
    \dfrac{(r+|x-\hat{x}|) ^{q-2}r}{|\nabla
    \varphi_{\hat{x}}(x)|}\right)^{p^{-}-2}(r+|x-\hat{x}|) ^{q-2} |\nabla 
    \varphi_{\hat{x}}(x)|^{p^{+}-2}r^{p^{-}(1-s)-1}\,dr \\& 
    \qquad  + \int_{|x-\hat{x}|}^\eta  
    \left(1+\dfrac{(r+|x-\hat{x}|) ^{q-2}r}{|\nabla 
    \varphi_{\hat{x}}(x)|}\right)^{p^{-}-2}(r+|x-\hat{x}|) ^{q-2} |\nabla 
    \varphi_{\hat{x}}(x)|^{p^{+}-2}r^{p^{-}(1-s)-1}\,dr \bigg) \\ & \quad 
    \leq C\bigg(\int_0^{|x-\hat{x}|}|x-\hat{x}| ^{q-2} |\nabla 
    \varphi_{\hat{x}}(x)|^{p^{-}-2}r^{p^{-}(1-s)-1}\,dr\\& \qquad \qquad 
    \qquad   +  \int_{|x-\hat{x}|}^\eta\left(\dfrac{r^{q-1}}{|\nabla 
    \varphi_{\hat{x}}(x)|} \right)^{p^{-}-2}r^{q-2}|\nabla 
    \varphi_{\hat{x}}(x)|^{p^{-}-2}r^{p^{-}(1-s)-1}\,dr \bigg) \\& \quad 
    \leq 
    C\left(\eta_0^{q(p^{-}-1)-sp^{-}} +\eta^{q(p^{-}-1)-sp^{-}} \right),
    \end{split}
   \]
    taking $q > \tfrac{sp^{-}}{p^{-}-1}$. We point out that $C$ depends on $\varepsilon$, but is independent of $x$, $\hat{x}$ and $\rho$.    
    
    We apply the arguments in  to get \eqref{eqq 8}. Hence, combining  
    \eqref{eqq 6} and  \eqref{eqq 8}, we finally get   \eqref{eqq 4}. 
    
    Now, we will apply Fatou's Lemma again to get
    \begin{equation*}
        \lim_{\rho \to 0}\mathcal{F}_\rho(x) \geq 
        \int_{\mathbb{R}^{n}}g\left(D_s u_\varepsilon\right) 
        \dfrac{1}{|x-y|^{n+s}}\,dy.
    \end{equation*}

    Observe that by symmetry,
    $$
        \mathcal{F}_\rho(x)= \int_{\mathbb{R}^{n}}\left[ 
        g\left(\dfrac{u_\varepsilon(x)-u_\varepsilon(y)}{(|x-y|+\rho)^{s}} 
        \right)-g\left(\dfrac{-\nabla 
        \varphi_{\hat{x}}(x)(y-x)\chi_{B_\eta(x)}(y)}{(|x-y|+\rho)^{s}} 
        \right)\right]\dfrac{1}{|x-y|^{n+s}}\,dy.
    $$

    We will prove that there is 
    $\mathcal{F}\in L^{1}(\mathbb{R}^{n})$ such that
    \begin{equation}\label{eqq 14}
        I(y)\coloneqq\left[g\left(\dfrac{u_\varepsilon(x)
        -u_\varepsilon(y)}{(|x-y|+\rho)^{s}} \right)-g\left(\dfrac{-\nabla 
        \varphi_{\hat{x}}(x)(y-x)\chi_{B_\eta(x)}(y)}{(|x-y|+\rho)^{s}} 
        \right)\right]\dfrac{1}{|x-y|^{n+s}} \geq -\mathcal{F}(y)
    \end{equation}
    in $\mathbb{R}^{n}$. 
    First, if $y \in \mathbb{R}^{n}\setminus B_\eta(x)$, then
    $$
        |I(y)| \leq Cg\left( \dfrac{1+|u_\varepsilon(y)|}{1+|y|^{s}}\right)
        \dfrac{1}{1+|y|^{n+s}}.
    $$
    Hence, the function 
    \begin{equation}\label{eqq 16}
        \mathcal{F}(y)\coloneqq Cg\left( 
        \dfrac{1+|u_\varepsilon(y)|}{1+|y|^{s}}\right)\dfrac{1}{1+|y|^{n+s}}, 
        \qquad y \in \mathbb{R}^{n}\setminus B_\eta(x),
    \end{equation}
    is in $L^{1}(\mathbb{R}^{n}\setminus B_\eta(x)).$ 
    Now, for $y \in B_\eta(x)$ and by Lemma \ref{lema:aux1}, we get reasoning 
    as in \eqref{eqq 50}, { and } \eqref{eqq 51} 
    that 
    \begin{equation}\label{eqq 13}
        \begin{split}
            I(y) &\geq -C\bigg(\dfrac{|\nabla 
            \varphi_{\hat{x}}(x)(y-x)+\tau_\infty 
            |x-y|^{2}|^{p^{-}-2}\tau_\infty |x-y|^{2}}{|x-y|^{n+sp^{-}}}\\
            & \quad + \dfrac{|\nabla \varphi_{\hat{x}}(x)(y-x)+\tau_\infty 
            |x-y|^{2}|^{p^{+}-2}\tau_\infty 
            |x-y|^{2}}{|x-y|^{n+sp^{+}}}\bigg) \in L^{1}(B_\eta(x)).
    \end{split}
    \end{equation} 
    Therefore, defining the function $\mathcal{F}$ as the right-hand side 
    in \eqref{eqq 13} over $B_\eta(x)$, and recalling   \eqref{eqq 16}, we 
    prove $\mathcal{F} \in L^1(\mathbb{R}^{n})$ and \eqref{eqq 14}. This 
    ends the proof of the lemma. 
\end{proof}

\subsection{Certain continuity properties}
The following lemmas will be useful for the proof of Theorem \ref{wk visc}.
\begin{lemma}\label{Continuidad D}
    Let $r>0$, $x_0\in \mathbb{R}^n$ and $F\in L_g(\mathbb{R}^n)$ Lipschitz
    in $B_r(x_0)$. For each $\varepsilon>0$, $0<\rho<r$ and 
    $\eta\in C^2_0(B_r(x_0))$ with $0\leq \eta\leq 1$, there exists 
    $\widetilde{\theta}=\widetilde{\theta}(\varepsilon, G, \eta)$ such that
    $F_\theta\coloneqq F+\theta\eta$ satisfies
    \begin{equation}\label{difDgs}
        \sup_{B_\rho(x_0)}|D_g^sF-D_g^sF_\theta|<\varepsilon\quad \text{for all }\, 0\leq \theta<\widetilde{\theta}. 
    \end{equation}
\end{lemma}

\begin{proof}
    Let $\varepsilon>0$, $0<\rho<r$ and $\eta\in C^2_0(B_r(x_0))$ 
    such that $0\leq \eta\leq 1$. Take $0<\theta<1$, $x\in B_\rho(x_0)$ and
    $0<\delta<r-\rho$ small enough to choose later. We have
    \begin{equation*}
        \begin{split}
            &|D^s_gF(x)-D^s_gF_\theta(x)| = 
            \left|\int_{\mathbb{R}^n}\left[G\left(|D_sF|\right)-G\left(|D_s
            F_\theta|\right)|\right]\frac{dy}{|x-y|^n}\right|\\&\leq 
            \int_{B_\delta(x)}\left|G\left(|D_sF|\right)-G\left(|D_sF_
            \theta|\right)\right|\frac{dy}{|x-y|^n} + 
            \int_{\mathbb{R}^n\setminus 
            B_\delta(x)}\left|G\left(|D_sF|\right)-G\left(|D_sF_\theta|
            \right)\right|\frac{dy}{|x-y|^n}\\& \leq 
            \int_{B_\delta(x)}\left|G\left(|D_sF|\right)
            \right|\frac{dy}{|x-y|^n} + 
            \int_{B_\delta(x)}\left|G\left(|D_sF_\theta|\right)\right|
            \frac{dy}{|x-y|^n} \\& \qquad + \int_{\mathbb{R}^n\setminus 
            B_\delta(x)}\left|G\left(|D_sF|\right)-G\left(|D_sF_\theta|
            \right)\right|\frac{dy}{|x-y|^n}\\& = I_1+I_2+I_3.
    \end{split}
    \end{equation*}

    Observe that, for $x\in B_\rho(x_0)$, $y\in B_\delta(x)$ and 
    $\delta<r-\rho$ it is true that $y\in B_{\rho+\delta}(x_0)\subset 
    B_r(x_0)$. Then, since $F$ is Lipschitz in $B_r(x_0)$
    $$
        |F(x)-F(y)|\leq K_F|x-y|.
    $$
    Taking $\delta<1$ and using \eqref{G product} we have for $I_1$
    \begin{equation}\label{I1}
        \begin{split}
            \int_{B_\delta(x)}G\left(|D_sF|\right)&\frac{dy}{|x-y|^n}  \leq
            \int_{B_\delta(x)}G\left(K_F|x-y|^{1-s}\right)
            \frac{dy}{|x-y|^n}  \\& \leq 
            CG(K_F)\int_{B_\delta(x)}
            \max\{|x-y|^{(1-s)p^{-}},|x-y|^{(1-s)p^+}\}\frac{dy}{|x-y|^n}\\
            & \leq CG(K_F)\left[\int_{B_\delta(x)}|x-y|^{(1-s)p^{-}-n}\,dy 
            + \int_{B_\delta(x)}|x-y|^{(1-s)p^+-n}\,dy\right]\\& 
            = C(G,K_F, n, s) 
            \left[\delta^{(1-s)p^-}+\delta^{(1-s)p^+}\right]. 
    \end{split}
    \end{equation}

    On the other hand, since $\eta\in C^2_0(B_r(x_0))$, 
    \begin{equation*}
        \begin{split}
            |F(x)+\theta\eta(x)-F(y)-\theta\eta(y)|&\leq 
            K_F|x-y|+\theta\sup_{z\in B_{\rho+\delta}(x_0)}|\nabla 
            \eta(z)||x-y|.
        \end{split}
    \end{equation*}Then, by \eqref{tineqG}, \eqref{G product}, and 
    recalling that $\theta<1$ we get for $I_2$ 
    \begin{equation}\label{I2}
        \begin{split}
            &\int_{B_\delta(x)}G\left(|D_sF_\theta|\right)
            \frac{dy}{|x-y|^n} \\& \leq \int_{B_\delta(x)}
            G\left(\frac{K_F|x-y|+\theta\sup_{z\in B_{\rho+\delta}(x_0)}
            |\nabla \eta(z)||x-y|}{|x-y|^s}\right)\frac{dy}{|x-y|^n} \\
            & \leq C\left[\int_{B_\delta(x)}
            \hspace{-.2cm}
            \dfrac{G\left(K_F|x-y|^{1-s}\right)}{|x-y|^n}dy+
            \int_{B_\delta(x)}\hspace{-.2cm}
            \dfrac{G\left(\theta\sup_{z\in 
            B_{\rho+\delta}(x_0)}|\nabla\eta(z)||x-y|^{1-s}\right)}
            {|x-y|^n}dy\right]\\& \leq C\left(G,K_F, 
            \sup_{B_{\rho+\delta}(x_0)}
            |\nabla \eta|,n,s\right)
            \left[\delta^{(1-s)p^{-}}+\delta^{(1-s)p^+}\right].
        \end{split}
    \end{equation}

    Finally, for $I_3$ we use the inequalities
    $$
        |G(a)-G(b)|\leq C|a-b|g(a+|a-b|), \quad a,b\geq 0,
    $$
    and $||a|-|b||\leq |a-b|$, and recalling that $0\leq \eta\leq 1$ 
    we get
    \begin{equation*}
        \begin{split}
            \int_{\mathbb{R}^n\setminus 
            B_\delta(x)}&\left|G\left(|D_sF|\right)-
            G\left(|D_sF_\theta|\right)\right|\frac{dy}{|x-y|^n} \\& \leq C
            \int_{\mathbb{R}^n\setminus
            B_\delta(x)}\frac{\theta|\eta(x)-\eta(y)|}{|x-y|^s}
            g\left(\frac{|F(x)-F(y)|+
            \theta|\eta(x)-\eta(y)|}{|x-y|^s}\right)\frac{dy}{|x-y|^n}\\& 
            \leq C2\theta \int_{\mathbb{R}^n\setminus
            B_\delta(x)}\frac{1}{|x-y|^s}
            g\left(\frac{2\theta+|F(x)|+|F(y)|}{|x-y|^s}\right)
            \frac{dy}{|x-y|^n}.
        \end{split}
    \end{equation*}Now, by \eqref{tineqgeq} we have 
    \begin{equation*}
        \begin{split}
            &\int_{\mathbb{R}^n\setminus 
            B_\delta(x)}\left|G\left(|D_sF|\right)-
            G\left(|D_sF_\theta|\right)\right|\frac{dy}{|x-y|^n} \\& \leq 
            C\theta \left[\int_{\mathbb{R}^n\setminus 
            B_\delta(x)}
            g\left(\frac{2\theta}{|x-y|^s}\right)
            \frac{dy}{|x-y|^{n+s}}\right. 
            + \int_{\mathbb{R}^n\setminus 
            B_\delta(x)}
            g\left(\frac{|F(x)|}{|x-y|^s}\right)\frac{dy}{|x-y|^{n+s}}
            \\& \qquad 
           + 
           \left.
           \int_{\mathbb{R}^n\setminus 
            B_\delta(x)}
            g\left(\frac{|F(y)|}{|x-y|^s}\right)
            \frac{dy}{|x-y|^{n+s}}\right]\\& = C\theta(I_3^1+I_3^2+I_3^3).
        \end{split}
    \end{equation*}
    For $I_3^1$ we use \eqref{H111} 
    \begin{equation}\label{I3_1}
        \begin{split}
            \int_{\mathbb{R}^n\setminus B_\delta(x)}&
            g\left(\frac{2\theta}{|x-y|^s}\right)\frac{dy}{|x-y|^{n+s}} \\
            & \leq C\int_{\mathbb{R}^n\setminus 
            B_\delta(x)}
            \max\left\{\frac{(2\theta)^{p^{-}-1}}{|x-y|^{n+sp^{-}}}, 
            \frac{(2\theta)^{p^+-1}}{|x-y|^{n+sp^+}}\right\}\,dy \\& \leq 
            C\left[(2\theta)^{p^{-}-1}\int_{\mathbb{R}^n\setminus 
            B_\delta(x)}\frac{dy}{|x-y|^{n+sp^{-}}} + 
            (2\theta)^{p^+-1}\int_{\mathbb{R}^n\setminus 
            B_\delta(x)}\frac{dy}{|x-y|^{n+sp^+}}\right]\\& \leq 
            C\left[(2\theta)^{p^{-}-1}\delta^{-sp^{-}} + 
            (2\theta)^{p^+-1}\delta^{-sp^+}\right] \\& \leq 
            C[\delta^{-sp^-}+\delta^{-sp^+}].
        \end{split}
    \end{equation}

    Reasoning in the same way and taking $A=\|F\|_{L^\infty(B_\rho(x_0))}$,
    we have for $I_3^2$
    \begin{equation}\label{I3_2}
        \int_{\mathbb{R}^n\setminus B_\delta(x)} 
        g\left(\frac{|F(x)|}{|x-y|^s}\right)\frac{dy}{|x-y|^{n+s}} \leq 
        C\max\left\{A^{p^{-}-1}, A^{p^{+}-1}\right\}[\delta^{-sp^-}
        +\delta^{-sp^+}]
    \end{equation}

    {Finally, take $R>0$ such that $B_{r}(x_0)\subset B_R$. Then, since 
    $\delta<1$ and $F\in L_g(\mathbb{R}^n)$, by Remark \ref{PVwelldefined} 
    for $B_{\delta}(x)\subset B_R$, we get for $I_3^3$}
    \begin{equation}\label{ineqLg}
    {
        \int_{\mathbb{R}^n\setminus B_\delta(x)}\frac{dy}{|x-y|^s}\leq 
        C\left(\frac{\delta}{1+R}\right)^{n+sp^{-}}\int_{\mathbb{R}^n
        \setminus B_{\delta}(x)}g\left(\frac{|F(y)|}{1+|y|^s}\right)
        \frac{dy}{1+|y|^{n+s}} \leq C.
        }
    \end{equation}
    {Observe that, because of the choice of $R$, the constant $C$ does not 
    depend on $x$.} Moreover,  since $\delta<1$, 
    $\delta^{-sp^-}+\delta^{-sp^+}>1$. Thus we can finally get for $I_3^3$ 
    \begin{equation}\label{I3_3}
        \int_{\mathbb{R}^n\setminus 
        B_\delta(x)}g\left(\frac{|F(y)|}{|x-y|^s}\right)
        \frac{dy}{|x-y|^{n+s}}\leq C[\delta^{-sp^-}+\delta^{-sp+}].
    \end{equation}
    Hence, by \eqref{I3_1}, \eqref{I3_2} and \eqref{I3_3} we have
    \begin{equation}\label{I30}
        \int_{\mathbb{R}^n\setminus 
        B_\delta(x)}\left|G\left(|D_sF|\right)-G\left(|D_sF_\theta|\right)
        \right|\frac{dy}{|x-y|^n}\leq 
        C\theta[\delta^{-sp^-}+\delta^{-sp^+}]
    \end{equation}
    and joining \eqref{I1}, \eqref{I2}, and \eqref{I30} we finally get 
    \begin{equation*}
        \begin{split}
        |D_g^sF(x)-D_g^sF_\theta(x)|&\leq 
        C(\delta^{(1-s)p^{-}}+\delta^{(1-s)p^+}+\theta\delta^{-sp^-}+\theta
        \delta^{-sp^+})  \\
        & = C(\delta^{-sp^-}(\delta^{p^-}+\theta)
        +\delta^{-sp^+}(\delta^{p^+}+\theta)).  
        \end{split}
    \end{equation*}
    Taking 
    \[
        \begin{split}
            &0<\delta<\min\left\{r-\rho, 1, 
            \left(\frac{\varepsilon}{2C}\right)^{1/(1-s)p^-}, 
            \left(\frac{\varepsilon}{2C}\right)^{1/(1-s)p^+}\right\}
            \text{ and }\\ 
            &0\leq \theta <
            \min\left\{\frac{\varepsilon\delta^{sp^+}}{2C}
            -\delta^{p^+},\frac{\varepsilon\delta^{sp^-}}{2C}-\delta^{p^-}, 1\right\},
        \end{split}
    \]
    we obtain \eqref{difDgs}.
\end{proof}

\begin{lemma}\label{ContPert}
    Let $B_r(x_0)\subset \Omega$ and $\psi\in C^2(B_r(x_0))\cap 
    L^\infty(\mathbb{R}^n)$ satisfying Definition \ref{dv} (iii) (a) or (b)
    with $\beta>\frac{sp^{-}}{p^{-}-1}$. Then for all $\varepsilon>0$ and 
    $\rho'>0$ there are $\theta'>0$, $\rho\in (0,\rho')$ and $\eta\in 
    C^2_0(B_{\rho/2}(x_0))$ with $0\leq \eta\leq 1$ and $\eta(x_0)=1$ such 
    that $\psi_{\theta} = \psi + \theta\eta$ satisfies
    \begin{equation}\label{supdifpert}
        \sup_{B_\rho(x_0)}|(-\Delta_g)^s\psi - (-\Delta_g)^s\psi_\theta|<\varepsilon
    \end{equation}
    for $0\leq\theta<\theta'.$
\end{lemma}

\begin{proof}
    Take $\varepsilon>0$, $\rho'>0$ and first assume that $\nabla\psi(x_0)\neq 0$. Then there is $\rho\in (0,\rho')$ such that $|\nabla\psi|>\tau$ in $B_{2\rho}(x_0)$ for some $\tau>0$. Now let $\eta\in C^2_0(B_{\rho/2}(x_0))$ with $0\leq \eta\leq 1$ and $\eta(x_0)=1$. Then there is $\theta''>0$ such that $|\nabla \psi_\theta|>\tau/2$ in $B_{2\rho}(x_0)$ when $0\leq \theta<\theta''$. Now observe that, for $x\in B_{\rho}(x_0)$,  $B_{\rho/2}(x)\subset B_{3\rho/2}(x_0)\subset\subset B_{2\rho}(x_0)\subset{\{d_{\psi_\theta}>0\}}$. Then we may apply Lemma \ref{lemma:PV1} to $\psi_\theta\in C^2(D)$ for $D=B_{3\rho/2(x_0)}$. Therefore we can take $\delta>0$ small enough such that, for every $x\in B_{\rho}(x_0)$ and $0\leq \theta<\theta''$ 
    \begin{equation}\label{PVepsilon1}
        \left|\text{P.V.} 
        \int_{B_\delta(x)}g\left(D_s\psi_\theta\right)\frac{dy}{|x-y|^{n+s}
        }\right|<\frac{\varepsilon}{4}.
    \end{equation}If $p^{-}>\frac{2}{2-s}$, we can get \eqref{PVepsilon1} 
    using Lemma \ref{lemma:PV1} whatever the value of $\nabla \psi(x_0)$ 
    is.
    Now consider the case 
    $1<p^{-}\leq \tfrac2{2-s}$, $|\nabla \psi(x_0)| = 0$ with $x_0$ an 
    isolated critical point and $\psi\in C^2_\beta(B_r(x_0))$. Then, 
    we can take $\rho>0$ small enough such that $|\nabla \psi|\neq 0$ in 
    $B_{3\rho}(x_0)\setminus \{x_0\}$. Let $\eta\in C^2_0(B_{\rho/2}(x_0))$
    such that $0\leq \eta\leq 1$, $\eta = 1$ in $B_{\rho/4}(x_0)$ and 
    $|D^2\eta|\leq Md_{\eta}^{\beta-2}$ for some $M>0$. Then $\nabla 
    \psi_{\theta}\neq 0$ in $B_{2\rho}(x_0)\setminus \{x_0\}$ for $\theta$ 
    small enough and, therefore, $d_{\psi} = d_{\psi_{\theta}}$ in 
    $B_{\rho}(x_0)$ for all such $\theta$. Also, since $\eta\in 
    C^2_0(B_{\rho/2}(x_0))$ and $|\nabla \psi|\neq 0$ in 
    $B_{3\rho}(x_0)\setminus \{x_0\}$, we may take $\theta$ small enough 
    such that $\theta|\nabla \eta|\leq \frac{1}{2}|\nabla \psi|$ in 
    $B_{\rho}(x_0)$. Then
    \begin{equation}\label{ineqpsi}
      \frac{1}{2}|\nabla \psi|\leq |\nabla \psi|-\theta|\nabla \eta|\leq 
      |\nabla \psi_{\theta}|\leq |\nabla \psi|+\theta|\nabla \eta|\leq 
      2|\nabla \psi|, \quad \text{in } B_{\rho}(x_0).  
    \end{equation}Moreover, since $d_\eta\leq d_\psi = d_{\psi_\theta}$ in 
    $B_{\rho}(x_0)$ and $\psi\in C^2_\beta(B_{\rho}(x_0))$ it holds
    \begin{equation}\label{ineqhespsitheta}
        |D^2\psi_{\theta}| \leq |D^2\psi|+\theta|D^2\eta|\leq 
        \|\psi\|_{C^2_\beta (B_\rho(x_0))}d_{\psi}^{\beta-2}+\theta M 
        d_{\eta}^{\beta-2}\leq cd_{\psi_\theta}^{\beta-2}.
    \end{equation}Therefore, since $d_{\psi} = d_{\psi_\theta}$ in 
    $B_{\rho}(x_0)$ and $\psi\in C^2_\beta(B_{\rho}(x_0))$ we get, by 
    \eqref{ineqpsi} and \eqref{ineqhespsitheta}, that $\psi_{\theta}\in 
    C^2_{\beta}(B_{\rho}(x_0))$. Then we may apply Lemma \ref{lemma:PV2} to
    find some $\delta\in (0, \rho)$ such that \eqref{PVepsilon1} also holds
    in this case.
    
    Now we proceed with the proof of \eqref{supdifpert}. Take $x\in 
    B_{\rho}(x_0)$. Then, by \eqref{PVepsilon1} and Lemma \ref{lema:aux1},
    taking $T=|D_s\psi|+|D_s\psi-D_s\psi_\theta|$
    we have
    \begin{equation}\label{Eqdif1}
        \begin{split}
            &|(-\Delta_g)^s\psi(x)-(-\Delta_g)^s\psi_{\theta}(x)|  = \left|\text{P.V.}\int_{\mathbb{R}^n}\frac{g(D_s\psi)-g(D_s\psi_{\theta})}{|x-y|^{n+s}}dy\right|\\& \leq \left|\text{P.V.}\int_{B_{\delta}(x)}\frac{g(D_s\psi)-g(D_s\psi_{\theta})}{|x-y|^{n+s}}dy\right|+\left|\text{P.V.}\int_{\mathbb{R}^n\setminus B_{\delta}(x)}\frac{g(D_s\psi)-g(D_s\psi_{\theta})}{|x-y|^{n+s}}dy\right|\\& \leq \frac{\varepsilon}{2} + \int_{\mathbb{R}^n\setminus B_{\delta}(x)}\frac{|g(D_s\psi)-g(D_s\psi_{\theta})|}{|x-y|^{n+s}}dy\\& \leq \frac{\varepsilon}{2}+C\int_{\mathbb{R}^n\setminus B_{\delta}(x)}\left[\frac{|D_s\psi-D_s\psi_{\theta}|}{|x-y|^{n+s}}\max\{T^{p^+-2}, T^{p^--2}\}\right]dy\\& \leq \frac{\varepsilon}{2} + C\int_{\mathbb{R}^n\setminus B_{\delta}(x)}\frac{|D_s\psi - D_s\psi_{\theta}|}{|x-y|^{n+s}}
            \left(T^{p^+-2}+T^{p^--2}\right)\,dy. 
        \end{split}
    \end{equation}
    Now we use the monotonicity of $(a+b)^{p-2}b$ for $a,b\geq 0$ and $p>1$, the fact that $|D_s\psi-D_s\psi_\theta|\leq \tfrac{2\theta}{|x-y|^{s}}$ and \eqref{Eqdif1} to get
    \begin{equation}\label{Eqdif2}
    \begin{split}
        &|(-\Delta_g)^s\psi(x)-(-\Delta_g)^s\psi_{\theta}(x)| \\& \leq   
        \frac{\varepsilon}{2} + C\theta\int_{\mathbb{R}^n\setminus 
        B_{\delta}(x)}\left[\frac{(|\psi(x)-\psi(y)|+
        2\theta)^{p^+-2}}{|x-y|^{n+sp^+}}+\frac{(|\psi(x)-\psi(y)|
        +2\theta)^{p^--2}}{|x-y|^{n+sp^-}}\right]dy\\
        & = \frac{\varepsilon}{2} +
        C\theta  \int_{\mathbb{R}^n\setminus 
        B_{\delta}(x)}\frac{(|\psi(x)-\psi(y)|
        +2\theta)^{p^+-2}}{|x-y|^{n+sp^+}}\,dy\\& + 
        C\theta\int_{\mathbb{R}^n\setminus 
        B_{\delta}(x)}\frac{(|\psi(x)
        -\psi(y)|+2\theta)^{p^--2}}{|x-y|^{n+sp^-}}\,dy \\
        &= \frac{\varepsilon}{2} + I + II. 
    \end{split}
    \end{equation}

    Next we estimate I assuming first that $1<p^+<2$. Then,  
    \begin{equation}\label{difp+<2}
         C\theta  \int_{\mathbb{R}^n\setminus 
         B_{\delta}(x)}\frac{(|\psi(x)-\psi(y)|+2\theta)^{p^+-2}}{|x-y|^{n+
         sp^+}}\,dy \leq  C\theta^{p^+-1}\delta^{-sp^+}  
         <\frac{\varepsilon}{4}.
    \end{equation}
    for all $\theta$ small enough. 
    Now suppose $p^+\geq 2$. Observe that $\theta^{p^+-1}<\theta$ for 
    $\theta<1$. Therefore, since $\psi\in L^\infty(\mathbb{R}^n)$ we can 
    estimate (I) for $\theta$ small enough as
    \begin{equation}\label{difp+>2}
        \begin{split}
         C\theta  \int_{\mathbb{R}^n\setminus 
         B_{\delta}(x)}\frac{(|\psi(x)-\psi(y)|
         +2\theta)^{p^+-2}}{|x-y|^{n+sp^+}}\,dy &\leq C\theta  
         \int_{\mathbb{R}^n\setminus 
         B_{\delta}(x)}
         \frac{(2\|\psi\|_{L^\infty(\mathbb{R}^n)}
         +2\theta)^{p^+-2}}{|x-y|^{n+sp^+}}\,dy \\& 
         \leq C\theta 
         (\|\psi\|_{L^\infty(\mathbb{R}^n)}^{p^+-2}+\theta^{p^+-2})
         \delta^{-sp^+}\\& \leq C\theta\delta^{-sp^+} \\& 
         <\frac{\varepsilon}{4}.
     \end{split}
    \end{equation}
    Reasoning in the same way, we can estimate II for all $\theta$ small 
    enough and finally use \eqref{Eqdif2} to get \eqref{supdifpert}.
\end{proof}

\begin{lemma}\label{contop}
    Let $B_r(x_0)\subset \Omega$ and $\psi\in C^2(B_r(x_0))\cap L_g(\mathbb{R}^n)$. We 
    also assume $\psi\in C^2_\beta(B_r(x_0))$ for some $\beta>\tfrac{sp^{-}}{p^{-}-1}$ 
    if $1<p^{-}\leq \tfrac2{2-s}$ and $\nabla \psi(x_0)=0$ with $x_0$ an isolated point 
    in $B_r(x_0)$. Then $(-\Delta_g)^s\psi$ is continuous in $B_r(x_0)$.
\end{lemma}

\begin{proof}
    Take $x\in B_r(x_0)$ and $\varepsilon>0$. First suppose $p^->\tfrac2{2-s}$ and choose
    $\delta>0$ such that $\overline{B_{\delta}(x)}\subset B_r(x_0)$. 
    Then, for $y\in B_{\delta}(x)$, there is $\delta'>0$ such that 
    $B_{\delta'}(y)\subset B_{\delta}(x)\subset\subset  \Omega$. Hence, 
    by Lemma \ref{lemma:PV1} there is $\rho>0$ such that 
     \begin{equation}\label{pvball}
        \left|\text{P.V.}\int_{B_\rho(y)}g\left(D_s\psi\right)\frac{dz}{|z-y|^{n+s}}\right|<\frac{\varepsilon}{4}
    \end{equation}whenever $|x-y|<\delta$. 
    
    Now assume $p^-\leq \tfrac2{2-s}$. If $\nabla \psi(x_0)\neq 0$ we take $r>0$ such 
    that $\nabla\psi(z)\neq 0$ for all $z\in B_r(x_0)$ and then $\nabla \psi(x)\neq 0$. 
    Hence, by continuity there is $\delta>0$ such that $\nabla\psi(y)\neq 0$ for all 
    $y\in \overline{B_\delta(x)}\subset B_r(x_0)$. Therefore, we can take $\delta'>0$ 
    such that $B_{\delta'}(y)\subset B_{\delta}(x)\subset\subset\{d_{\psi}>0\}$ and we 
    may apply again Lemma \ref{lemma:PV1} to find $\rho>0$ such that \eqref{pvball} holds
    for all $y\in B_{\delta}(x)$. 
    
    If on the contrary we have $p^{-}\leq \tfrac2{2-s}$ and $\nabla \psi(x_0) = 0$ we 
    choose $r>0$ such that $\nabla \psi(z)\neq 0$ for all 
    $z\in B_r(x_0)\setminus \{x_0\}$. Then, if $x\neq x_0$, $\nabla \psi(x)\neq 0$ 
    and we proceed as we did before. Now, if $x=x_0$,  $|x_0-y|<\delta$ implies 
    $d_\psi(y)<\delta$ and we also have $\psi\in C^2_\beta(B_r(x_0))$. Take $0<\delta<1$ 
    such that \eqref{pvball} holds for the first two cases and also impose $\delta<r/2$. 
    Then $B_{\delta}(y)\subset B_{2\delta}(x_0)\subset B_{r}(x_0)$ for all 
    $y\in B_{\delta}(x_0)$ and $d_{\psi}(y)<\delta$. 
    Hence we may use Lemma \ref{lemma:PV2} to find $\rho>0$ such that  \eqref{pvball} 
    also holds in this case for all $y\in B_{\delta}(x)$.
    
    Now we may suppose that $|x-y|<\rho/3$. Thus, using \eqref{tineqgeq} we have 
    \begin{equation}\label{boundint}
        \begin{split}
            g(D_s\psi)
            \frac{\mathcal{X}_{\mathbb{R}^n\setminus B_\rho(y)}(z)}{|y-z|^{n+s}} 
            & \leq c
            \left[g\left(\frac{|\psi(y)|}{|y-z|^s}\right)
            +g\left(\frac{|\psi(z)|}{|y-z|^s}\right)\right]
            \frac{\mathcal{X}_{\mathbb{R}^n\setminus B_\rho(y)}(z)}{|y-z|^{n+s}}\\
            & \leq c\left[g\left(\frac{\|\psi\|_{L^\infty(B_{\rho/3}(x))}}{\rho^s}\right)
            +g\left(\frac{|\psi(z)|}{|y-z|^s}\right)\right]
            \frac{\mathcal{X}_{\mathbb{R}^n\setminus B_\rho(y)}(z)}{|y-z|^{n+s}}.
        \end{split}
    \end{equation}
    If we take $\rho<1$, then the right hand side of \eqref{boundint} is integrable in 
    $\mathbb{R}^n$ by Remark \ref{PVwelldefined} since $\psi\in L_g(\mathbb{R}^n)$. 
    $g\left(\frac{|\psi(z)|}{|y-z|^s}\right)\frac{1}{|y-z|^{n+s}}$ belongs to 
    $L^1(\mathbb{R}^n\setminus B_\rho(y))$.
    Hence, using dominated convergence theorem and the continuity of 
    $g\left(\frac{\psi(\cdot)-\psi(z)}{|\cdot-z|^s}\right)\frac{1}{|\cdot-z|^{n+s}}$ in 
    $\mathbb{R}^n\setminus\{z\}$ we obtain
    \begin{equation}\label{limitytox}
        \int_{\mathbb{R}^n\setminus 
        B_\rho(y)}g\left(\frac{\psi(y)-\psi(z)}{|y-z|^s}\right)\frac{dz}{|y-z|^{n+s}}\to
        \int_{\mathbb{R}^n\setminus 
        B_\rho(x)}g\left(\frac{\psi(x)-\psi(z)}{|x-z|^s}\right)\frac{dz}{|x-z|^{n+s}} 
    \end{equation}
    as $y\to x$. Then there is $\delta>0$ such that 
    \begin{equation}\label{difoutball}
        \left|\int_{\mathbb{R}^n\setminus 
        B_\rho(y)}g\left(\frac{\psi(y)-\psi(z)}{|y-z|^s}\right)\frac{dz}{|y-z|^{n+s}}-
        \int_{\mathbb{R}^n\setminus 
        B_\rho(x)}g\left(\frac{\psi(x)-\psi(z)}{|x-z|^s}\right)\frac{dz}{|x-z|^{n+s}}
        \right|<\frac{\varepsilon}{2} 
    \end{equation}whenever $|x-y|<\delta$. 
    
    Finally by \eqref{pvball} and \eqref{difoutball} we get 
    $$
        |(-\Delta_g)^s\psi(x)-(-\Delta_g)^s\psi(y)|<\varepsilon
    $$
    if $|x-y|$ is small enough.
\end{proof}

\section{Proof of Theorem \ref{visc wk}} \label{visc wk t}
    In this section, we prove that viscosity solutions to \eqref{problem} are also weak 
    solutions. {We will follow in the proof some calculations from \cite{KKL} and \cite{BM} adapted to the Orlicz framework.}
    \begin{proof}[Proof of Theorem \ref{visc wk}] 
        By Lemma \ref{infconvsol}, $u_\varepsilon$ is a viscosity supersolution of 
        $$
            (-\Delta_g)^{s}w=f_\varepsilon(x, w, D_g^{s}w) \quad \textnormal{in }\Omega_{r(\varepsilon)},
        $$
        and hence it satisfies
        $$
            (-\Delta_g)^{s}u_\varepsilon \geq f_\varepsilon(x, u_\varepsilon, D_g^{s}u_\varepsilon)\quad a.e. \textnormal{ in }\Omega_{r(\varepsilon)}.
        $$
        By Lemmas \ref{lemma 1 weak point} and \ref{lemma 2 weak point},
        \begin{equation}\label{wk u ep}
            \int\int_{Q_{\Omega_{r(\varepsilon)}}} 
            g\left(D_su_\varepsilon \right)
            \dfrac{\psi(x)-\psi(y)}{|x-y|^{n+s}}\,dx\,dy 
            \geq \int_{\Omega_{r(\varepsilon)}}f_\varepsilon(x, u_\varepsilon, D_g^{s}u_\varepsilon)\,\psi\,dx.
        \end{equation}
        Hence, $u_\varepsilon$ is a weak supersolution in $\Omega_{r(\varepsilon)}$. 
        Let now $\varphi \in C_0^{\infty}(\Omega)$, non-negative with 
        $K\coloneqq $supp$\,\varphi$. 
        Then,  $K \subset \Omega_{r(\varepsilon)}$ 
        for all $\varepsilon$ small enough.  
        The goal is to  take the limit as $\varepsilon \to 0$ in \eqref{wk u ep}. 
        
        Let $\xi \in C_0^{\infty}(\Omega_{r(\varepsilon)})$, $0 \leq \xi \leq 1$ such 
        that $\xi = 1$ in $K' \subset K''$, where $K'$ contains $K$ and $K''$ is a 
        compact set containing the support of $\xi$. By Proposition \ref{caccio},
        \begin{equation}
            \begin{split}
                \int_{K'}\int_{\mathbb{R}^{n}}G\left(|D_s u_\varepsilon|\right)\,d\mu 
                & = \int_{K'}\int_{\mathbb{R}^{n}}G\left(|D_s 
                u_\varepsilon|\right)G(\xi(x))\,d\mu \\ &  \leq C\left[ 
                G(\text{osc}\,u_\varepsilon)\left(\int_{K''}\int_{\mathbb{R}^{n}} 
                G\left(|D_s\xi| \right)\,d\mu + \gamma_{\infty, \varepsilon}\right) + 
                \text{osc}(u_\varepsilon)\right],
            \end{split}
        \end{equation}
        with
        $$
            \gamma_{\infty,\varepsilon}\coloneqq
            \max_{[-\|u_\varepsilon\|_{L^\infty(\mathbb{R}^n)}, 
            \|u_\varepsilon\|_{L^\infty(\mathbb{R}^n)}]}\gamma(t).
        $$
        Since $u_\varepsilon$ is increasing as $\varepsilon \to 0^{+}$ and 
        $u \in L^{\infty}(\mathbb{R}^{n})$, we have
        $$
            \text{osc}(u_\varepsilon) \leq \sup_{\mathbb{R}^{n}}u - 
            \inf_{\mathbb{R}^{n}}u_{\varepsilon_0},
        $$
        for all $\varepsilon \leq \varepsilon_0$, 
        and also 
        $\|u_\varepsilon\|_{L^{\infty}(\mathbb{R}^{n})} \leq 
        \|u\|_{L^{\infty}(\mathbb{R}^{n})}$. Thus,
        \begin{equation}\label{bd G u}
            \int_{K'}\int_{\mathbb{R}^{n}}G\left( |D_s u_\varepsilon|\right) \,d\mu 
            \leq C \quad \text{ for }\varepsilon \leq \varepsilon_0.
        \end{equation}
        Hence, up to a subsequence, it holds
        \begin{equation}\label{strong c}
            u_\varepsilon \to u \quad \text{in } L^{G}(K')
        \end{equation}and
        \begin{equation}\label{weak conv mu}
            D_su_\varepsilon \rightharpoonup D_su \quad L^{G}_\mu(K' \times 
            \mathbb{R}^{n}).
        \end{equation}
        For further reference observe that \eqref{bd G u} and Lemma \ref{G g} imply
        $$
            \int_{K'}\int_{\mathbb{R}^{n}}\tilde{G}(g(D_s u_\varepsilon))\,d\mu \leq C\int_{K'}\int_{\mathbb{R}^{n}}G(D_s u_\varepsilon)\,d\mu \leq C
        $$
        for all $\varepsilon$. Thus,
        \begin{equation}\label{wk conv g tilde}
            g(D_s u_\varepsilon) \rightharpoonup g(D_s u) 
            \quad \text{in }L^{\tilde{G}}_\mu(K' \times \mathbb{R}^{n}).
        \end{equation}
    
        Consider now
        $$
            \psi(x)\coloneqq (u(x)-u_\varepsilon(x))\theta(x),
        $$
        where $\theta \in C_0^{\infty}(\Omega)$, 
        supp$\,\theta \, \subset K'$, $\theta \in [0, 1]$ 
        and $\theta = 1$ in $K \subset K'$. Observe that
        $$
            \psi(x)-\psi(y)= \theta(x)\left(u(x)-u_\varepsilon(x)-(u(y)
            -u_\varepsilon(y)) \right)+
            (\theta(x)-\theta(y))(u(y)-u_\varepsilon(y)).
        $$
        Then, by \eqref{wk u ep}
        \begin{equation}\label{iIa}
            \begin{split}
                & \int_{K'}f_\varepsilon(x, u_\varepsilon, D_g^{s}u_\varepsilon)\psi\,dx 
                \leq \int\int_{Q_{K'}}g\left(D_su_\varepsilon \right) D_s\psi\, d\mu \\ 
                & \qquad \quad = \int\int_{Q_{K'}}g\left( D_s u_\varepsilon 
                \right)\dfrac{(\theta(x)-\theta(y))(u(y)-u_\varepsilon(y))}{|x-y|^{s}}\,
                d\mu \\ 
                &  + \int_{K'}\int_{\mathbb{R}^{n}}g\left( D_s u \right)
                \dfrac{\theta(x)(u(x)-u(y)-(u_\varepsilon(x)
                -u_\varepsilon(y)))}{|x-y|^{s}}\,d\mu \\ 
                & \qquad \quad - \int_{K'}\int_{\mathbb{R}^{n}}\left[ 
                g(D_su_\varepsilon)-g(D_su)\right]\dfrac{\theta(x)(u(x)-u(y)
                -(u_\varepsilon(x)-u_\varepsilon(y)))}{|x-y|^{s}}\,d\mu.
        \end{split}
        \end{equation}
        Now, by H\"{o}lder's inequality
         \begin{equation}\label{first integral}
            \begin{split}
                &\int\int_{Q_{K'}}g\left( D_s u_\varepsilon
                \right)\dfrac{(\theta(x)-\theta(y))(u(y)-u_\varepsilon(y))}{|x-y|^{s}}
                \,d\mu \\
                & \qquad \leq C\|g\left( D_s u_\varepsilon 
                \right)\|_{L_\mu^{\tilde{G}}(Q_{K'})}\bigg\| 
                \dfrac{(\theta(x)-\theta(y))(u(y)-u_\varepsilon(y))}{|x-y|^{s}}
                \bigg\|_{L_\mu^{G}(Q_{K'})} \\
                &  \qquad\leq 
                C\bigg\|\dfrac{(\theta(x)-\theta(y))(u(y)
                -u_\varepsilon(y))}{|x-y|^{s}}\bigg\|_{L_\mu^{G}(Q_{K'})} .
            \end{split}
        \end{equation}
        Also, suppose without loss of generality  that
        \begin{equation}
            \begin{split}
                \min\left\lbrace \bigg\| 
                    \dfrac{(\theta(x)-\theta(y))(u(y)
                    -u_\varepsilon(y))}{|x-y|^{s}} 
                    \bigg\|_{L_\mu^{G}(Q_{K'})}^{p^{+}}\right.&,
                    \left.\bigg\| 
                    \dfrac{(\theta(x)-\theta(y))(u(y)-u_\varepsilon(y))}{|x-y|^{s}}
                    \bigg\|_{L_\mu^{G}(Q_{K'})}^{p^{-}} \right\rbrace \\ 
                    &  = \bigg\| \dfrac{(\theta(x)-\theta(y))
                    (u(y)-u_\varepsilon(y))}{|x-y|^{s}}
                    \bigg\|_{L_\mu^{G}(Q_{K'})}^{p^{-}}.
             \end{split}
        \end{equation}Then, by Lemma \ref{comp norm modular},
        \begin{equation}\label{int zero}
            \begin{split}
                \bigg\| \dfrac{(\theta(x)-\theta(y))(u(y)-u_\varepsilon(y))}{|x-y|^{s}} 
                &\bigg\|_{L_\mu^{G}(Q_{K'})}  \leq 
                \left[ \Phi_{\mu,G}\left(\dfrac{(\theta(x)-\theta(y))(u(y)
                -u_\varepsilon(y))}{|x-y|^{s}} \right)\right]^{1/p^{-}} \\
                &  = \left[\int\int_{Q_{K'}}G
                \left( \dfrac{(\theta(x)-\theta(y))
                (u(y)-u_\varepsilon(y))}{|x-y|^{s}} \right) \,d\mu \right]^{1/p^{-}}. 
             \end{split}
        \end{equation}
        
        Next, we will show that the last integral converges to zero. 
        Observe that {\eqref{G product}} and the uniform  global 
        handedness of $u_\varepsilon$ imply that
        \begin{equation}
            G\left(\dfrac{(\theta(x)-\theta(y))(u(y)-u_\varepsilon(y))}{|x-y|^{s}}\right)
            \leq C(2\|u\|_{L^{\infty}(\mathbb{R}^{n})}
            {+1})^{p^{+}}G(|D_s\theta|).
        \end{equation}
        Since 
        $$
            G(|D_s\theta|) \leq C\max\left\lbrace |D_s\theta|^{p^{+}}, 
            |D_s\theta|^{p^{-}}\right\rbrace \leq C\left( |D_s\theta|^{p^{+}}+ 
            |D_s\theta|^{p^{-}}\right),$$
        we obtain by the smoothness of $\theta$ that 
        $$
            G\left(\dfrac{(\theta(x)-\theta(y))
            (u(y)-u_\varepsilon(y))}{|x-y|^{s}}\right)  \in L^{1}_\mu(Q_{K'}).
        $$
        
        By dominated convergence theorem, the last integral in \eqref{int zero} goes to 
        zero and hence recalling \eqref{first integral}, it holds
        \begin{equation}\label{I1a}
             \int\int_{Q_{K'}}g\left( D_s u_\varepsilon 
             \right)\dfrac{(\theta(x)-\theta(y))(u(y)-u_\varepsilon(y))}{|x-y|^{s}}
             \,d\mu \to 0 \quad \text{as }\varepsilon \to 0^{+}. 
        \end{equation}
        Now, we treat the  following integral in \eqref{iIa}:
        \begin{equation}\label{second inte}
            \int_{K'}\int_{\mathbb{R}^{n}}g\left( D_s u 
            \right)\dfrac{\theta(x)(u(x)-u(y)
            -(u_\varepsilon(x)-u_\varepsilon(y)))}{|x-y|^{s}}\,d\mu.
        \end{equation}
        Observe that by Lemma \ref{G g},
        $$
            g(D_su)\theta \in L_\mu^{\tilde{G}}(K' \times \mathbb{R}^{n}).
        $$
        Hence, by \eqref{weak conv mu} it holds
        \begin{equation}\label{I2a}
            \int_{K'}\int_{\mathbb{R}^{n}}g
            \left( D_s u\right)\dfrac{\theta(x)(u(x)-u(y)
            -(u_\varepsilon(x)-u_\varepsilon(y)))}{|x-y|^{s}}\,d\mu \to 0 
            \quad \text{as }\varepsilon \to 0^{+}. 
        \end{equation}
    
        Finally, we consider the integral
        \begin{equation}\label{third inte}
            \int_{K'}\int_{\mathbb{R}^{n}}
            \left[ g(D_su_\varepsilon)-g(D_su)\right]\dfrac{\theta(x)(u(x)-u(y)
            -(u_\varepsilon(x)-u_\varepsilon(y)))}{|x-y|^{s}}\,d\mu.
        \end{equation}
        By the convexity of $G$, we have
        $$
            G(|D_s u|) \leq G\left( \bigg| \frac{D_su + 
            D_s u_\varepsilon}{2}\bigg|\right) 
            + g(|D_s u|)\frac{D_s u}{|D_s u|}
            \left( \frac{D_su - D_s u_\varepsilon}{2}\right)
        $$
        and
        $$
            G(|D_s u_\varepsilon|) \leq G\left( \bigg| 
            \frac{D_su + D_s u_\varepsilon}{2}\bigg|\right) 
            + g(|D_s u_\varepsilon|)\frac{D_s u_\varepsilon}{|D_s u_\varepsilon|}\left( \frac{D_su_\varepsilon - D_s u}{2}\right).
        $$
        Adding the above expressions gives
        \begin{equation}\label{G conv 1}
            \begin{split}
            \frac{1}{2}\left[ g(D_s)-g(D_s u_\varepsilon)\right]
            (D_su - D_s u_\varepsilon) &\geq G(|D_s u|)+ G(|D_S u_\varepsilon|) 
            - G\left( \bigg| \frac{D_su + D_s u_\varepsilon}{2}\bigg|\right) \\ 
            & \geq  G\left( \bigg| \frac{D_su - D_s u_\varepsilon}{2}\bigg|\right),
        \end{split}
        \end{equation}
        where the last inequality follows again by convexity of $G$ 
        (see \cite[Lemma 2.9]{L}). 
        Hence, by \eqref{iIa}, \eqref{I1a}, \eqref{I2a}, and \eqref{G conv 1}, 
        we obtain
        \begin{equation}
            \begin{split}0 
                & \leq \lim_{\varepsilon \to 0}\int_{K'}\int_{\mathbb{R}^{n}}
                \left[ g(D_su_\varepsilon)-g(D_su)\right]
                \dfrac{\theta(x)(u(x)-u(y)-(u_\varepsilon(x)
                -u_\varepsilon(y)))}{|x-y|^{s}}\,d\mu \\ 
                & \leq \limsup_{\varepsilon \to 0}
                \left(-\int_{K'}f_\varepsilon(x, u_\varepsilon, D_g^{s}u_\varepsilon) \right).
            \end{split}
        \end{equation}
     
        From the assumption \eqref{growth f}, we have
        $$
            -\int_{K'}f_\varepsilon(x, u_\varepsilon, D_g^{s}u_\varepsilon) \,dx 
            \leq \gamma_\infty \int_{K'}\tilde{G}^{-1}
            \left( |D_g^{s}u_\varepsilon|\right)(u-u_\varepsilon)\theta\,dx 
            + \|\phi\|_{L^{\infty}(K')}\int_{K'}(u-u_\varepsilon)\theta\,dx.
        $$
        Now, by H\"{o}lder's inequality,
        $$
            \int_{K'}\tilde{G}^{-1}\left( |D_g^{s}u_\varepsilon|\right)
            (u-u_\varepsilon)\theta\,dx \leq 
            \|\tilde{G}^{-1}(|D_g^{s}u_\varepsilon|)\|_{L^{\tilde{G}}(K')}
            \|u-u_\varepsilon\|_{L^{G}(K')}.
        $$
        By \eqref{strong c}, which in particular holds for $B=G$, 
        $$
            \|u-u_\varepsilon\|_{L^{G}(K')} \to 0 \quad \text{as }\varepsilon \to 0.
        $$
        On the other hand,
        \begin{equation}
            \begin{split}
                \|\tilde{G}^{-1}(D_g^{s}u_\varepsilon)\|_{L^{\tilde{G}}(K')} 
                & \leq \Phi_{\tilde{G}}(\tilde{G}^{-1}
                (D_g^{s}u_\varepsilon))^{1/p^{-}} + 
                \Phi_{\tilde{G}}(\tilde{G}^{-1}(D_g^{s}u_\varepsilon))^{1/p^{+}}  
                \\ & = \left(\int_{K'}\int_{\mathbb{R}^{n}}
                G\left(D_s u_\varepsilon
                \right)\,d\mu\right)^{1/p^{+}}
                +\left(\int_{K'}\int_{\mathbb{R}^{n}}
                G\left(D_s u_\varepsilon \right)\,d\mu\right)^{1/p^{-}},
            \end{split}
        \end{equation}
        and the last terms are uniformly bounded by \eqref{weak conv mu}. 
        Therefore 
        $$ 
            \limsup_{\varepsilon \to 0}\left(-\int_{K'}f_\varepsilon(x, u_\varepsilon, D_g^{s}u_\varepsilon)\,dx \right) =0
        $$
        and thus
        \begin{equation}\label{I3}
            \lim_{\varepsilon \to 0}\int_{K'}\int_{\mathbb{R}^{n}}
            \left[ g(D_su_\varepsilon)-g(D_su)\right]
            \dfrac{\theta(x)(u(x)-u(y)
            -(u_\varepsilon(x)-u_\varepsilon(y)))}{|x-y|^{s}}\,d\mu=0.
        \end{equation} 
        Now, from \eqref{G conv 1}, we get
        \[
            \begin{split}
                \frac{1}{2}\left[ g(D_s)-g(D_s u_\varepsilon)\right]
                (D_su - D_s u_\varepsilon)  \geq G \left( \frac{|D_s u-D_s u_\varepsilon|}{2}\right)   \geq \left(\frac{1}{2} \right)^{p^{-}}G(|D_s u-D_s u_\varepsilon|),
            \end{split}
        \]
        where the last inequality follows from  \eqref{G product}. 
        In this way, \eqref{I3} and Lemma \ref{conv f} imply that we may pass to the 
        limit as $\varepsilon \to 0$ in \eqref{wk u ep} to get
        $$
            \int \int_{Q_K}g(D_s u)D_s\psi\,d\mu \geq \int_K 
            f(x, u, D_g^{s}u)\psi\,dx,
        $$
        where in the left-hand side we use \eqref{wk conv g tilde}. 
        This ends the proof of the  theorem. 
    \end{proof}


\section{Proof of Theorem \ref{wk visc}}\label{wk visc t}
    Lastly, we show that a weak solution is also a viscosity solution.
    
    \begin{proof}[Proof of Theorem \ref{wk visc}] 
        We proceed by contradiction. Assume that $u \in L^{\infty}(\mathbb{R}^{n})$ 
        is a continuous weak supersolution, but it is not a viscosity supersolution. 
        Hence, according to Remark \ref{u bdd},  there are a point $x_0 \in \Omega$ 
        and a test function $\psi \in C^{2}(B_r(x_0))\cap L_g(\mathbb{R}^{n})$, 
        that we may take equal to $u$ outside $B_r(x_0)$ and thus in 
        $L^\infty(\mathbb{R}^n)$ since $u$ is bounded, satisfying (iii) 
        from Definition \ref{dv} and
        \begin{equation}
            (-\Delta_g)^{s}\psi (x_0) < f(x_0, \psi(x_0), D_g^{s}\psi(x_0)).
        \end{equation}
        By Lemma \ref{Continuidad D}, the mapping
        $$
            x \to f(x, \psi(x), D_g^{s}\psi (x))
        $$
        is continuous in $B_r(x_0)$. 
        Also, by Lemma \ref{contop},  there exist $\delta$ and $0 < r_1 < r$ 
        such that
        \begin{equation}\label{eqq 100}
            (-\Delta_g)^{s} \psi(x) 
            \leq f(x, u, D_g^{s}\psi(x)) - \delta, \quad x \in B_{r_1}(x_0).
        \end{equation}
        Since $\psi \in C^{2}(B_r(x_0))\cap L^\infty(\mathbb{R}^n)$, 
        by Lemma \ref{ContPert}, for every $\varepsilon > 0$ and $\rho >0$, 
        there exist $0 < \theta_1 = \theta(\varepsilon, \rho)$, $0 < r_2 < \rho$ 
        and $\eta \in C_0^{2}(B_{r_2/2}(x_0))$, $\eta \in [0, 1]$, with $\eta(x_0)=1$ 
        such that
        $$
            \sup_{B_{r_2}(x_0)}|(-\Delta_g)^{s}\psi - (-\Delta_g)^{s}
                (\psi+\theta \eta) | < \varepsilon
        $$
        for every $0 < \theta < \theta_1$. Taking $\varepsilon = \delta/2$ and 
        $\rho = r_1$ and using the Lipschitz assumption on $f$, we get by Lemma \ref{Continuidad D} and  \eqref{eqq 100} that
        \begin{equation}\label{eqq 101}
        \begin{split}
                (-\Delta_g)^{s}(\psi+\theta \eta)(x) &\leq (-\Delta_g)^{s}\psi(x) + 
                \frac{\delta}{2} 
                \leq f(x, u(x), D_g^{s}\psi(x)) -\frac{\delta}{2} \\ 
                & \leq f(x, u(x), D_g^{s}(\psi+\theta\eta)(x))
        \end{split}
        \end{equation}
        for $x \in B_{r_2/2}$ and 
        $0 < \theta < \min\left\lbrace \theta_1(\delta/2, r_1), 
        \tilde{\theta}(\delta/(2K)), r_2 \right\rbrace$, 
        where $\tilde{\theta}$ is given 
        by Lemma \ref{Continuidad D} and $K$ is the Lipschitz constant of $f$. 
        Let now
        $$
            \tilde{f}(x, v)\coloneqq f(x, u(x), v).
        $$
        Then, \eqref{eqq 101} also holds weakly in $B_{r_2/2}(x_0)$ and hence 
        $\psi+\theta \eta$ is a weak subsolution of the equation
        $$
            (-\Delta_g)^{s}v(x)= \tilde{f}(x, D_g^{s}v(x)), \qquad x \in B_{r_2/2}(x_0).
        $$
        Since $u$ is a weak supersolution of the same equation and 
        $\psi+\theta \eta \leq u$ in $\mathbb{R}^{n}\setminus B_{r_2/2}(x_0)$, 
        by the {(CPP)}, we get $\psi+\theta \eta \leq u$ 
        in $\mathbb{R}^{n}$. In particular,
        $$
            u(x_0) \geq \psi(x_0)+ \theta\eta(x_0) =  \psi(x_0)+ \theta > \psi(x_0)
        $$
        which is a contradiction with the fact that $u(x_0)= \psi(x_0)$. 
    \end{proof}

    \subsection{A comparison principle} 
        In this section, we provide a comparison principle for non-homogeneous $g-$Laplace 
        equations. It is worth to  point out that the general comparison principle  needed 
        in the proof of Theorem \ref{wk visc} remains an open problem. 
        \begin{theorem}
            Let $f=f(x, r)$ be non-increasing in $r$. Assume that $u$ and $v$ 
            are weak sub and supersolutions, respectively,  of
            $$
                (-\Delta_g)^s w= f(x, w) \quad \text{in }\Omega,
            $$
            with $u \leq v$ in $\mathbb{R}^n\setminus \Omega$. 
            Then $v\geq u$  in  $\mathbb{R}^n.$
\end{theorem}
\begin{proof}
    Using $\psi=(u-v)^+$ as a test function for $u$ and $v$, and subtracting give
    \begin{equation}\label{comp1}
        \int_{\mathbb{R}^n}\int_{\mathbb{R}^n}\left( g\left(D_s v \right)-g\left(D_s u \right)\right)D_s\psi \,d\mu \geq \int_\Omega (f(x, v)-f(x, u))\psi(x)\,dx. 
    \end{equation}Observe that the latter integral is indeed over the set $\{u\geq v\}$ and hence, since $f(x, r)$ is non-increasing in $r$, the integral is greater or equal than $0$. The left-hand side of \eqref{comp1} may be treated as in \cite[Lemma C.4]{BSV} to get that $(u-v)^+ =0$. Hence, comparison follows. 
\end{proof}

\appendix
\section{Some inequalities for Young functions}

\begin{lemma}\label{tineqg}
    Suppose that 
    \begin{equation}\label{H11Bis}
    1< p^{-} \leq \dfrac{tg(t)}{G(t)} \leq p^{+} <\infty.
    \end{equation}
    with $1< p^-<p^+<\infty$. 
    Then, there is a positive constant $C$ such that 
    \begin{equation}\label{tineqgeq}
        g(s+t)\leq C(g(s)+g(t))
    \end{equation}for all $s,t\geq 0$.
\end{lemma}

\begin{proof}
    By Theorem 4.1 from \cite{KR}, \eqref{H11Bis} implies that $G$ 
    satisfies the $\Delta_2$-condition, that is, there exists a constant 
    $\textbf{C}>2$ such that
    \begin{equation}\label{Delta2cond}
        G(2t)\leq \textbf{C}G(t), \quad t>0.
    \end{equation}
    The convexity of $G$ and \eqref{Delta2cond} give
    \begin{equation}\label{tineqG}
        G(s+t)\leq \frac{\textbf{C}}{2}(G(s)+G(t)), \quad s,t>0. 
    \end{equation}
    Then, by \eqref{H11Bis} and \eqref{tineqG} we get, for 
    $s,t>0$
    \begin{equation*}
        \begin{split}
            g(s+t) & \leq \frac{p^+}{s+t}G(s+t) \leq 
            \frac{p^+\textbf{C}}{2(s+t)}(G(s)+G(t))\\& \leq 
            \frac{p^+\textbf{C}}{2p^{-}(s+t)}(sg(s)+tg(t))  = 
            \frac{p^+\textbf{C}}{2p^{-}}\left(\frac{s}{s+t}g(s)+\frac{t}{s+t}g(t)\right) 
            \leq \frac{p^+\textbf{C}}{2p^{-}}(g(s)+g(t)).
        \end{split}
    \end{equation*}
    Finally, recalling that $g(0)=0$ we obtain \eqref{tineqgeq} for all $s,t\geq 0$.
\end{proof}

    \begin{lemma}\label{lema:aux1}
        Let $p^->1.$ Then there is a positive constant $C$ 
	    such that 
	    \[
	        |g(a+b)-g(b)|\le C\max\left\{(|b|+|a|)^{p^--2},(|b|+|a|)^{p^+-2}\right\}
	        |a|
	    \]
	    for any $a,b\in\mathbb{R}.$
	\end{lemma}
	\begin{proof}
	    Observe that
	    \[
	        |g(a+b)-g(b)|=\left|a\int_0^1 g^\prime(b+ta) dt\right|
	        \le|a|\int_0^1 |g^\prime(b+ta)| dt
	        \le|a|\int_0^1 g^\prime(|b+ta|) dt.
	    \]
	    Then, by \eqref{H1} and \eqref{H111}, there is a positive constant 
	    $C_1=C_1(p^-,p^+)$ such that
	    \[
	        	|g(a+b)-g(b)|\le C_1 |a|\max
	        	\left\{\int_{0}^1 |a+tb|^{p^--2} dt,\int_{0}^1 |a+tb|^{p^+-2} dt.
	        	\right\}
	    \]
	    Then, by \cite[Lemma 3.2]{KKL}, there is a positive constant 
	    $C_2=C_2(p^-,p^+)$ such that
	    \[
	        |g(a+b)-g(b)|\le C_2\max\left\{(|b|+|a|)^{p^--2},(|b|+|a|)^{p^+-2}\right\}
	        |a|.
	    \]
	\end{proof}
\section*{Acknowledgements}
 M. L. de Borb\'on and P. Ochoa  have been supported by CONICET and Grant B080, UNCUYO, Argentina.

\end{document}